\documentclass[11pt,twoside,reqno]{amsart}
\usepackage[utf8]{inputenc}
\usepackage[T1]{fontenc}
\usepackage{import}
\usepackage{newclude}
\usepackage{anysize}
\usepackage{amsmath}
\usepackage{amsthm}
\usepackage{amsmath,amscd}
\usepackage[utf8]{inputenc}
\usepackage{amssymb}
\usepackage{scalerel}
\usepackage{stmaryrd}
\usepackage{wasysym}
\usepackage{mathrsfs}
\usepackage[pagebackref=true,hypertexnames=false]{hyperref}
\renewcommand*{\backref}[1]{}
\renewcommand*{\backrefalt}[4]{({\tiny%
   \ifcase #1 Not cited.%
         \or Cited on page~#2.%
         \else Cited on pages #2.%
   \fi%
   })}
\usepackage{cleveref}
\usepackage{enumitem}
\usepackage{dsfont}
\usepackage{soul}
\hypersetup{colorlinks=true,linkcolor=black,citecolor=black}
\usepackage{color}
\usepackage[all]{xy}
\usepackage{float}
\usepackage{bm}
\usepackage{mathtools}
\usepackage{makecell}
\usepackage{thmtools}
\usepackage{setspace}
\usepackage{comment}
\usepackage{tikz}
\usepackage{tikz-cd}
\usepackage{mathdots}
\usepackage{graphicx}
\usepackage[myheadings]{fullpage}
\usepackage{csquotes}

\DeclareFontFamily{U}{mathb}{}

\DeclareFontShape{U}{mathb}{m}{n}{
   <-5.5>  mathb5
   <5.5-6.5> mathb6
   <6.5-7.5> mathb7
   <7.5-8.5> mathb8
   <8.5-9.5> mathb9
   <9.5-11>  mathb10
   <11->     mathb12
}{}

\DeclareSymbolFont{mathb}{U}{mathb}{m}{n}

\DeclareMathSymbol{\lefttorightarrow}{3}{mathb}{"FC}
\DeclareMathSymbol{\righttoleftarrow}{3}{mathb}{"FD}

\makeatletter
\newcommand{\acts}{%
  \mathrel{\mathpalette\@acts\righttoleftarrow}}
\newcommand{\actedby}{%
  \mathrel{\mathpalette\@acts\lefttorightarrow}}
\newcommand{\@acts}[2]{\reflectbox{$\m@th#1#2$}}
\makeatother

\numberwithin{equation}{section}
\newcommand\mtop{.95in}
\newcommand\mbottom{.95in}
\newcommand\mleft{1in}
\newcommand\mright{1in}
\usepackage[top = \mtop, bottom = \mbottom, left = \mleft, right=\mright]{geometry}

\DeclareMathOperator{\Mat}{Mat}

\newtheorem{thm}{Theorem}[section]

\newtheorem{corollary}[thm]{Corollary}

\newtheorem{prop}[thm]{Proposition}

\newtheorem{lemma}[thm]{Lemma}

\theoremstyle{definition}
\newtheorem{defi}[thm]{Definition}

\newcommand\reallywidehat[1]{%
\savestack{\tmpbox}{\stretchto{%
  \scaleto{%
    \scalerel*[\widthof{\ensuremath{#1}}]{\kern-.6pt\bigwedge\kern-.6pt}%
    {\rule[-\textheight/2]{1ex}{\textheight}}
  }{\textheight}%
}{0.5ex}}%
\stackon[1pt]{#1}{\tmpbox}%
}
\DeclareSymbolFont{bbold}{U}{bbold}{m}{n}
\DeclareSymbolFontAlphabet{\mathbbold}{bbold}

\makeatletter
\def\@tocline#1#2#3#4#5#6#7{\relax
  \ifnum #1>\c@tocdepth 
  \else
    \par \addpenalty\@secpenalty\addvspace{#2}%
    \begingroup \hyphenpenalty\@M
    \@ifempty{#4}{%
      \@tempdima\csname r@tocindent\number#1\endcsname\relax
    }{%
      \@tempdima#4\relax
    }%
    \parindent\z@ \leftskip#3\relax \advance\leftskip\@tempdima\relax
    \rightskip\@pnumwidth plus4em \parfillskip-\@pnumwidth
    #5\leavevmode\hskip-\@tempdima
      \ifcase #1
       \or\or \hskip 1em \or \hskip 2em \else \hskip 3em \fi%
      #6\nobreak\relax
    \hfill\hbox to\@pnumwidth{\@tocpagenum{#7}}\par
    \nobreak
    \endgroup
  \fi}
\makeatother

\newcommand{\R}{\mathbb{R}}
\newcommand{\Z}{\mathbb{Z}}
\newcommand{\Q}{\mathbb{Q}}

\newcommand{\C}{\mathbb{C}}
\newcommand{\F}{\mathbb{F}}

\newcommand{\E}{\mathbb{E}}

\newcommand{\mc}{\mathcal}

\renewcommand{\L}{\Lambda}

\DeclareMathOperator{\new}{new}
\DeclareMathOperator{\Disc}{Disc}

\DeclareMathOperator{\Mod}{Mod}

\DeclareMathOperator{\poly}{poly}

\DeclareMathOperator{\GL}{GL}

\DeclareMathOperator{\un}{un}
\DeclareMathOperator{\ab}{ab}
\DeclareMathOperator{\tame}{tame}

\title[Eigenvalue Distribution of $p$-adic Random Matrices Among Algebraic Extensions]{Eigenvalue Distribution of $p$-adic Random Matrices Among Algebraic Extensions, with an Analogue for $p$-adic Random Polynomials}
\author{Jiahe Shen}

\date{\today}

\begin{document}

\thanks{The author thanks Mehtaab Sawhney for asking the question, and Roger Van Peski for reading the draft and providing helpful comments. The author acknowledges support from Ivan Corwin's NSF grant DMS-2246576 and Simons Investigator grant 929852.}

\maketitle

\begin{abstract}
We study the distribution of eigenvalues of Haar-random matrices over $\Z_p$
among algebraic extensions of $\Q_p$. Our results give $p$-adic analogues of
the real-eigenvalue counting results of Edelman-Kostlan-Shub \cite{edelman1994many} for the real Ginibre ensemble, but with a different degree behavior: while real eigenvalues form only a vanishing proportion in the real Ginibre ensemble, $p$-adic eigenvalues are asymptotically evenly distributed among possible extension degrees. We also show that the maximal unramified extension $\Q_p^{\un}$ captures all but a bounded expected number of eigenvalues, and that the expected number of eigenvalues outside $\Q_p^{\un}$ has a finite positive limit with an explicit upper bound.

The proof uses correlation function formulas from the author's previous joint
work with Van Peski \cite{shen2026eigenvalues}, together with uniform estimates over varying finite extensions. We also prove analogous results for roots of random Haar polynomials over $\Z_p$, using the correlation function formulas of Caruso \cite{caruso2022zeroes}. These polynomial results are $p$-adic analogues of the real-root counting results of Edelman-Kostlan \cite{edelman1995many}, again with behavior
different from the real setting.
\end{abstract}



\tableofcontents

\section{Introduction}\label{sec: Intro}

\subsection{Main results}

The goal of the present paper is to study how the eigenvalues of a Haar-random
matrix over $\Z_p$ are distributed among algebraic extensions of $\Q_p$ as the
matrix size tends to infinity.  This question is a $p$-adic analogue of a
classical phenomenon in real random matrix theory.  For example, in the real
Ginibre ensemble, the entries are independent real Gaussian random variables,
and one may ask how many eigenvalues are real rather than genuinely complex.
If $E_n$ denotes the expected number of real eigenvalues of an $n\times n$ real
Ginibre matrix, Edelman-Kostlan-Shub \cite{edelman1994many} proved that
\begin{equation}\label{eq: expected number of real eigenvalues}
E_n=\sqrt{\frac{2n}{\pi}}
\left(1-\frac{3}{8n}
+O\left(\frac{1}{n^2}\right)\right)
+\frac12.
\end{equation}
Thus only a vanishing proportion of the eigenvalues are real in expectation.

Motivated by this contrast between real and imaginary eigenvalues, this paper
studies the corresponding question for $p$-adic random matrices. Over the real numbers, the algebraic extension structure is very simple: the only proper finite extension of $\R$ is $\C$.  Thus, for real random matrices, the basic algebraic distinction is whether an eigenvalue lies in $\R$ or in $\C\setminus \R$.  The situation over $\Q_p$ is much richer.  The field $\Q_p$ has infinitely many finite extensions, and these extensions are organized by arithmetic invariants
such as their degree, ramification index, inertia degree, and discriminant.
Therefore, for a Haar-random matrix $A\in\Mat_n(\Z_p)$, one can ask not only how many eigenvalues lie in $\Q_p$ itself, but also how the eigenvalues are distributed among the various algebraic extensions of $\Q_p$.

This question is also natural from the viewpoint of algebraic number theory.
Certain large algebraic extensions of $\Q_p$, such as the maximal unramified extension, play central roles in the arithmetic of local fields. These fields reflect structural features of non-archimedean local fields which have no
direct counterpart over $\R$ or $\C$. In this paper, we study the distribution of eigenvalues of $p$-adic random matrices among algebraic extensions from two perspectives: first by organizing
eigenvalues according to the degree of the extension they generate, and second by studying how many eigenvalues lie outside the maximal unramified extension
$\Q_p^{\un}$.

We now make this question precise. Throughout the paper, $p$ is a prime number, $\Q_p$ is the field of $p$-adic numbers, and $\bar\Q_p$ is the algebraic closure of $\Q_p$. Let $\Mat_n(\Z_p)$ be the space of $n\times n$ matrices with entries in $\Z_p$.

\begin{defi}\label{defi: counting eigenvalues}
Let $U \subset \bar\Q_p$ be a subset of $\bar\Q_p$, the algebraic closure of $\Q_p$. We denote by
$$Z_{U,n}=Z_{U,n}(A),\quad A\in\Mat_n(\Z_p)$$
the number of eigenvalues of $A$ that lie in $U$.
\end{defi}

For example, one may take $U=\Z_p$.  Then, for $A\in\Mat_n(\Z_p)$, the function
$Z_{\Z_p,n}=Z_{\Z_p,n}(A)$ counts the number of eigenvalues of $A$ which lie in $\Z_p$.  Throughout the paper, we regard
$\Mat_n(\Z_p)$ as a probability space equipped with the additive Haar probability measure, and we only consider subsets $U$ for which $Z_{U,n}$ is measurable and thus can be regarded as a random variable over $\Mat_n(\Z_p)$.

We will often want to count eigenvalues according to the finite extension of
$\Q_p$ that they generate. This is analogous to the archimedean setting in the following sense. If $A$ is a real matrix, then all its eigenvalues lie in $\C$, but one is usually interested in separating the eigenvalues which already lie in $\R$ from those which lie in
$\C\setminus \R$.  In the $p$-adic setting, all eigenvalues of a matrix in $\Mat_n(\Z_p)$ lie in $\overline{\Q}_p$, but there are many intermediate finite extensions of $\Q_p$. To avoid counting the same eigenvalue in several larger
extensions, we isolate the part of each extension consisting of elements which
do not come from a proper subextension.

Thus, given a finite extension $K/\Q_p$, we write $K^{\new}$ for the subset of $K$ consisting of elements which do not lie in any proper subextension of $K$.  The unit ball in $K$, equivalently $\overline{\Z}_p\cap K$, is the ring of integers $\mathcal O_K$. Moreover, every eigenvalue of a matrix $A\in\Mat_n(\Z_p)$
which lies in $K$ in fact lies in $\mathcal O_K$.  We therefore set
\[
\mathcal O_K^{\new}:=K^{\new}\cap \mathcal O_K.
\]
We first study the distribution of eigenvalues from the perspective of
extension degree. This viewpoint is closer in spirit to the archimedean question discussed above: instead of asking whether an eigenvalue is real, one asks how large the algebraic extension generated by an eigenvalue is. More precisely, for $1\le r\le n$, one may ask how many eigenvalues of a Haar-random matrix $A\in\Mat_n(\Z_p)$ generate an extension of $\Q_p$ of degree at most $r$.

Some special cases of this question are already understood. For example, \cite[Section 9]{shen2026eigenvalues} gives explicit formulas for the limiting
expected number of eigenvalues lying in $\Z_p$ and in quadratic extensions. In
particular, a surprising special case is that the expected number of eigenvalues
lying in $\Z_p$ is exactly equal to $1$ for every matrix size $n$. At the other extreme, every eigenvalue of an $n\times n$ matrix generates an extension of degree at most $n$, so the total number of such eigenvalues is trivially $n$. The natural question is what happens between these two extremes. For instance, how many eigenvalues generate extensions of degree at most $\sqrt n$, or at most $\alpha n$ for some fixed constant $\alpha\in(0,1)$? The following theorem answers this question for any cutoff $r_n$ which tends to infinity while remaining a nontrivial distance from $n$.

\begin{thm}\label{thm: zeros degree under rn}
Suppose the positive integer sequence $(r_n)_{n\ge 1}$ satisfies 
\begin{equation*}
\lim_{n\rightarrow\infty} r_n=+\infty,\quad
\lim_{n\rightarrow\infty}(n-r_n)=+\infty.
\end{equation*}
Then, we have
\begin{equation*}
\sum_{[K:\Q_p]\le r_n}\E[Z_{\mc{O}_K^{\new},n}]=r_n+\sum_{m=1}^\infty\left(1-(p^{-m};p^{-1})_\infty\right)+o(1).
\end{equation*}
Here we use the standard $q$-Pochhammer notation
$(p^{-m};p^{-1})_\infty:=\prod_{k=0}^{\infty}\left(1-p^{-m-k}\right)$.
\end{thm}

For example, taking $r_n=\lceil \alpha n\rceil$ where $\alpha\in(0,1)$ is a
fixed constant, \Cref{thm: zeros degree under rn} implies that
\begin{equation*}
\lim_{n\rightarrow\infty}\frac{1}{n}
\sum_{[K:\Q_p]\le\lceil \alpha n\rceil}
\E\bigl[Z_{\mathcal O_K^{\new},n}\bigr]
=\alpha.
\end{equation*}
This is the sense in which the eigenvalues are asymptotically evenly distributed among possible extension degrees: a cutoff at degree approximately $\alpha n$ captures approximately an $\alpha$-proportion of the eigenvalues in expectation. This is in sharp contrast with the real Ginibre ensemble.
Indeed, for an $n\times n$ real Ginibre matrix, the asymptotic
\eqref{eq: expected number of real eigenvalues} shows that the expected number
of real eigenvalues is asymptotic to $\sqrt{2n/\pi}$, and hence real
eigenvalues form only a vanishing proportion of all eigenvalues.

Let us briefly explain the main input behind the proof. The starting point is
the correlation function formula from the author's previous joint work with Van Peski \cite[Lemma 1.13]{shen2026eigenvalues}. Here the one-point
correlation function is an explicit density function whose integral over a
$p$-adic region gives the expected number of eigenvalues lying in that region.
Applied to a finite extension $K/\Q_p$, this formula gives such a density for eigenvalues lying in $\mathcal O_K^{\new}$. The main additional estimate in this paper is to control these integrands uniformly as both the degree of $K$ and the matrix size $n$ vary. In particular, for the unique unramified extension of degree $r$, this expectation is close to $1$ whenever both $r$ and $n-r$
are large. On the other hand, the total contribution of ramified extensions of large degree can be controlled using the discriminant factor in the correlation function together with the mass formula recalled later in \Cref{lem: serre mass}. Combining these two ingredients gives the extension-degree asymptotic in \Cref{thm: zeros degree under rn}.

The same calculation also leads to another natural way to organize the eigenvalues, namely according to whether they lie in the maximal unramified extension $\Q_p^{\un}$. Since the maximal unramified
extension $\Q_p^{\un}$ contains exactly one unramified extension of each degree, the preceding estimate suggests that
$\Q_p^{\un}$ should capture almost all eigenvalues in expectation. Our second main result confirms this expectation. More precisely, as the matrix size tends to infinity, the expected number of eigenvalues which fail to lie in $\Q_p^{\un}$ remains bounded.

\begin{thm}\label{thm: outside un}
The limit
\begin{equation*}
\lim_{n\rightarrow\infty}\E[Z_{\bar\Q_p\backslash\Q_p^{\un},n}]=\lim_{n\rightarrow\infty}\left(n-\E[Z_{\Q_p^{\un},n}]\right)
\end{equation*}
exists. Moreover, this limit is positive but no larger than
$\frac{p(4p-3)}{(p-1)^3}$.
\end{thm}

This result is somewhat surprising from the viewpoint of algebraic extensions.
Indeed, the maximal unramified extension is only a very small part of $\overline{\Q}_p$. For each fixed positive integer $r$, there are many degree $r$ extensions of $\Q_p$, but only one of them is unramified, namely
\begin{equation*}
\Q_p[\zeta_{p^r-1}],
\end{equation*}
where $\zeta_{p^r-1}$ is a primitive $(p^r-1)$-st root of unity. \Cref{thm: outside un} says that, nevertheless, the contribution of all ramified extensions together is bounded in expectation as $n\to\infty$.

One can also consider other distinguished infinite algebraic extensions of
$\Q_p$, for example the maximal abelian extension $\Q_p^{\ab}$ and the maximal tamely ramified extension $\Q_p^{\tame}$. Since both of these fields contain
$\Q_p^{\un}$, we have the inequalities
\begin{equation*}
Z_{\Q_p^{\ab},n}\ge Z_{\Q_p^{\un},n},
\qquad
Z_{\Q_p^{\tame},n}\ge Z_{\Q_p^{\un},n}.
\end{equation*}
Therefore \Cref{thm: outside un} already implies that the expected numbers of
eigenvalues outside $\Q_p^{\ab}$ and outside $\Q_p^{\tame}$ are bounded as $n\to\infty$. Moreover, following the same approach, one can prove the stronger statement that the corresponding limits exist and are finite.

Let us also mention two natural directions for further study. First, one may ask for more refined asymptotics in the above results. In the archimedean setting, Edelman-Kostlan-Shub \cite{edelman1994many} obtained an asymptotic expansion for the
expected number of real eigenvalues of the real Ginibre ensemble to arbitrary
order. By contrast, our arguments for \Cref{thm: outside un} and
\Cref{thm: zeros degree under rn} are designed to prove convergence and
identify the leading-order behavior, but they do not capture finer asymptotic
terms. Obtaining such refinements would require substantially sharper uniform estimates for the relevant correlation functions, which seem out of reach with the methods used here.

Another natural direction is to study $p$-adic random matrix ensembles with additional symmetries, for example Hermitian or alternating matrices over
$p$-adic rings. This differs in an essential way from the archimedean
setting: for Hermitian matrices, all eigenvalues are real, and for real
alternating matrices, all nonzero eigenvalues are purely imaginary. In the $p$-adic setting, however, the analogous symmetry conditions need not force all eigenvalues to lie in a fixed finite extension of $\Q_p$. Eigenvalues may still generate extensions of varying degrees.  It is therefore natural to expect suitable analogues of the phenomena studied in this paper: almost all eigenvalues should lie in the maximal unramified extension
$\Q_p^{\un}$, and, in expectation, the eigenvalues should be asymptotically evenly distributed from the perspective of extension degree.

\subsection{Random Haar polynomial}

Apart from the eigenvalues of a random matrix, we also consider the closely
related model of random polynomials over $\Z_p$. Let 
$$\Omega_n(\Z_p):=\{a_nx^n+a_{n-1}x^{n-1}+\cdots+a_1x+a_0:a_0,\ldots,a_n\in\Z_p\}$$ 
be the space of polynomials over $\Z_p$
of degree at most $n$.

\begin{defi}
For a subset $U\subset \overline{\Q}_p$, we define
\begin{equation*}
Z_{U,n}^{\poly}=Z_{U,n}^{\poly}(P),
\qquad
P\in\Omega_n(\Z_p),
\end{equation*}
to be the number of roots of $P$ which lie in $U$, counted with multiplicity.
Thus $Z_{U,n}^{\poly}$ is a function on $\Omega_n(\Z_p)$.
\end{defi}

In the following, we regard $\Omega_n(\Z_p)$ as a probability space by taking the coefficients $a_0,\ldots,a_n$ to be independent random variables
distributed according to the additive Haar probability measure on $\Z_p$. As in the matrix setting, we only consider subsets $U$ for which $Z_{U,n}^{\poly}$ is measurable and thus can be regarded as a random variable on $\Omega_n(\Z_p)$.

There is, however, one important difference between the polynomial setting and the matrix setting.  If $A\in\Mat_n(\Z_p)$, then every eigenvalue of $A$ which lies in a finite extension $K/\Q_p$ must in fact lie in $\mc O_K$.  By contrast, a root of a polynomial in $\Omega_n(\Z_p)$ may have $p$-adic absolute value larger than $1$.  Thus, for a finite extension $K/\Q_p$, roots which generate $K$ need not lie in $\mc O_K^{\new}$.  Accordingly, in the polynomial setting we write $Z_{K^{\new},n}^{\poly}$ for the number of
roots of $P_n$ which lie in $K^{\new}$.

This polynomial model leads to the same two questions as the random matrix
model: how many roots lie in the maximal unramified extension, and how many
roots generate extensions of degree at most a given cutoff? The following two
results give analogues of \Cref{thm: zeros degree under rn} and \Cref{thm: outside un}.

\begin{thm}\label{thm: poly zeros degree under rn}
Suppose the positive integer sequence $(r_n)_{n\ge 1}$ satisfies
\begin{equation*}
\lim_{n\rightarrow\infty} r_n=+\infty,
\qquad
\lim_{n\rightarrow\infty}(n-r_n)=+\infty.
\end{equation*}
Then, we have
\begin{equation*}
\sum_{[K:\Q_p]\le r_n}
\E\bigl[Z_{K^{\new},n}^{\poly}]=r_n+\frac{1}{p-1}+o(1).
\end{equation*}
\end{thm}

\begin{thm}\label{thm: poly outside un}
The limit
\begin{equation*}
\lim_{n\rightarrow\infty}
\E\bigl[
Z_{\overline{\Q}_p\setminus\Q_p^{\un},n}^{\poly}
\bigr]
\end{equation*}
exists. Moreover, this limit is positive but no larger than $\frac{(4p-3)}{(p-1)^2}$.
\end{thm}

The proofs of \Cref{thm: poly zeros degree under rn} and \Cref{thm: poly outside un} use the correlation functions for
random Haar polynomials obtained by Caruso \cite{caruso2022zeroes}. The argument is parallel to the random matrix case: one estimates the corresponding one-point correlation functions uniformly as both the degree of the extension and the degree of the polynomial vary.

Let us again point out the contrast with the archimedean analogue. For a random
real polynomial of degree $n$ with independent Gaussian coefficients,
Edelman-Kostlan \cite{edelman1995many} proved that the expected number of real
roots is asymptotic to $\frac{2}{\pi}\log n$. Thus only a vanishing proportion of its roots lie in the base field $\R$. By contrast, the $p$-adic Haar polynomial model exhibits the same qualitative
behavior as the $p$-adic random matrix model, where the roots are asymptotically evenly distributed from the perspective of extension degree.

\subsection{Outline of the paper}

In \Cref{sec: Preliminaries}, we collect the correlation function formula from
\cite{shen2026eigenvalues}, the needed lattice counting estimate, and the local field facts used to sum over ramified extensions. In \Cref{sec: Proof}, we prove the random matrix results by establishing uniform estimates for finite extensions and applying them to unramified extensions and degree cutoffs. Finally, in \Cref{sec: The random polynomial case}, we use the correlation function formula from
\cite{caruso2022zeroes} to prove the analogous results for random Haar polynomials.

\section{Preliminaries}\label{sec: Preliminaries}

In this section, we collect the preliminary results needed for the proof of the main theorems stated in \Cref{sec: Intro}. We also recall a few standard facts about finite extensions of $\Q_p$. For a more detailed introduction to $p$-adic fields and their extensions, see the textbook of Neukirch \cite{neukirch2013algebraic}. For background on the lattice interpretation of $p$-adic orbital integrals, which underlies the correlation formula used here, we refer to the note of Yun \cite{yun2013orbital}.

\begin{defi}\label{defi: Den and distance}
Let $K/\Q_p$ be an extension of degree $r$, $\mc{O}_K$ its ring of integers, and $x \in \mc{O}_K^{\new}$. Then we define
$$\Mod_{\Z_p[x]}=\{M\subset K\mid M \text{ is a $\Z_p[x]$-module, also a $\Z_p$-lattice of rank $r$}\}.$$ 
We have a canonical action of the group $K^\times$ on $\Mod_{\Z_p[x]}$ by multiplication, and similarly for its subgroup $\Lambda = \{\pi^n: n \in \Z\}$ where $\pi$ is a uniformizer of $K$. This latter action has finitely many orbits, and we denote the number by $\#(\Lambda \backslash \Mod_{\Z_p[x]})$. 
\end{defi}

For example, when $x\in\Z_p$, then we have $\#(\L\backslash \Mod_{\Z_p[x]})=1$. When $x$ generates a quadratic extension, the explicit expression is given in \cite[Proposition 9.10]{shen2026eigenvalues}. In particular, when $x$ is a generator of $\mc{O}_K$, i.e., $\mc{O}_K=\Z_p[x]$, we have $\#(\L\backslash \Mod_{\Z_p[x]})=1$.

\begin{prop}\label{prop: prelimiting correlation function}
Let $K/\Q_p$ be a finite extension of degree $r=ef\le n$, where $e$ is the ramification index and $f$ is the inertia degree. Let $U \subset \mc{O}_{K}^{\new}$ be any measurable set. Then, we have
$$\E[Z_{U,n}]=\int_U\rho^{(n)}_{K}(x)dx,$$
where
\begin{equation}
\rho_{K}^{(n)}(x)=\prod_{i=0}^{r-1}(1-p^{-n+i})\cdot\frac{||\Disc_{K/\Q_p}||}{1-p^{-f}} \cdot \frac{\#(\L\backslash \Mod_{\Z_p[x]})}{\#(\mc{O}_{K}/\Z_p[x])^2} \cdot \E||\det(Z(A))||.
\end{equation}
Here, the integration (and definition of measurability of $U$) is with respect to the Haar probability measure on $\mc{O}_K$, $Z$ is the minimal polynomial of $x$, and $A\in\Mat_{n-r}(\Z_p)$ is distributed by the additive Haar measure.
\end{prop}

\begin{proof}
We apply \cite[Lemma 1.13]{shen2026eigenvalues} in the special case
$m=1$ and $K_1=K$.  In the notation of that paper, the relevant correlation
function is denoted
$$
\rho^{(n)}_{K_1,\ldots,K_m}(x_1,\ldots,x_m).
$$
Thus, in the present one-field case, the correlation function in
\cite[Lemma 1.13]{shen2026eigenvalues} is exactly the function denoted here by
$\rho_K^{(n)}(x)$.

Let $Z=Z_x$ be the minimal polynomial of $x$ over $\Q_p$.
By \cite[Lemma 1.13]{shen2026eigenvalues}, for $r=[K:\Q_p]$ and
$n\ge r$, we have
$$
\rho^{(n)}_K(x)=\prod_{i=0}^{r-1}(1-p^{-n+i})\cdot||\Delta_\sigma(x)||\cdot V(x)\cdot\E||\det(Z(A))||,
$$
where $A\in \Mat_{n-r}(\Z_p)$ is Haar distributed.

It remains to translate the notation in this formula into the notation of
the present paper.  In \cite[Definition 1.6]{shen2026eigenvalues}, the
function $V$ is defined by
$$
V(x)=\frac{1}{1-p^{-r/e}}||\Delta_\sigma(x)||
\#(\Lambda\backslash \Mod_{\Z_p[x]})=\frac{1}{1-p^{-f}}
||\Delta_\sigma(x)||
\#(\Lambda\backslash \Mod_{\Z_p[x]}).
$$
Substituting this into the previous formula gives
$$
\rho^{(n)}_K(x)=\prod_{i=0}^{r-1}(1-p^{-n+i})\frac{||\Delta_\sigma(x)||^2}{1-p^{-f}}\#(\Lambda\backslash \Mod_{\Z_p[x]})\E||\det(Z(A))||.
$$

We now rewrite the factor $||\Delta_\sigma(x)||^2$. In the notation of \cite{shen2026eigenvalues}, $\Delta_\sigma(x)$ is the Vandermonde factor formed from the Galois conjugates of $x$, and the usual discriminant-index formula (see, for instance, page 10 of \cite{caruso2022zeroes}) gives
$$
||\Delta_\sigma(x)||^2=||\Disc_{K/\Q_p}||\cdot\#(\mathcal O_K/\Z_p[x])^{-2}.
$$
Hence
$$
\rho^{(n)}_K(x)=\prod_{i=0}^{r-1}(1-p^{-n+i})\cdot\frac{||\Disc_{K/\Q_p}||}{1-p^{-f}}\cdot\frac{\#(\Lambda\backslash \Mod_{\Z_p[x]})}
{\#(\mathcal O_K/\Z_p[x])^2}
\cdot\E||\det(Z(A))||.
$$
This is exactly the formula claimed in the statement.

Finally, \cite[Lemma 1.13]{shen2026eigenvalues} gives the corresponding integral formula for the one-point correlation function: for every measurable set $U\subset \mathcal O_K^{\new}$,
$$
\E[Z_{U,n}]=\int_U \rho^{(n)}_K(x)dx,
$$
where $dx$ is the additive Haar probability measure on $\mathcal O_K$.
This completes the proof.
\end{proof}

From the correlation function formula in \Cref{prop: prelimiting correlation function}, we see that an upper bound for the correlation function requires an upper bound for the lattice-counting term $\#\bigl(\Lambda\backslash \Mod_{\Z_p[x]}\bigr)$. The following lemma provides the needed estimate.

\begin{lemma}[{\cite[Lemma 6.5]{shen2026eigenvalues}}]\label{lem: estimate of orbital integral}
Let $K/\Q_p$ be a finite extension. Suppose $x\in\mathcal{O}_K^{\new}$ satisfies $\F_p[x\pmod{p}]=\F_{p^d}$, where $x\pmod{p}$ is the image of $x$ in the residue field of $K$. When $\Z_p[x]$ is a proper subset of $\mathcal{O}_K$, we have
$$\frac{\#(\Lambda\backslash\Mod_{\Z_p[x]})}{\#(\mathcal{O}_K/\Z_p[x])^2}\le p^{-d}+2p^{-2d}\le 1.$$
\end{lemma}

For an integer $r\ge 1$, let $\F_{p^r}$ denote the finite field of order $p^r$, and set
\begin{equation}\label{eq: G_r}
G_r:=\#\{\alpha\in\F_{p^r}:\F_p[\alpha]=\F_{p^r}\}.
\end{equation}
Thus $G_r$ is the size of the subset of elements of $\F_{p^r}$ which generate $\F_{p^r}$ over $\F_p$. This quantity will appear repeatedly when we estimate integrals of correlation functions over unramified extensions. Heuristically, for an unramified extension of degree $r$, the correlation function is largest on those residue classes whose reductions generate
$\F_{p^r}$ over $\F_p$. Therefore, when integrating the correlation function over the unramified extension, the generator residue classes give the main contribution, and the factor $G_r/p^r$ naturally appears. The following elementary formula for $G_r$ is essentially contained in Caruso
\cite[(4.2)]{caruso2022zeroes}.

\begin{lemma}\label{lem: order of generator}
For every integer $r\ge 1$, we have
$$
G_r=\sum_{d|r}\mu\left(\frac{r}{d}\right)p^d.
$$
Here $\mu$ is the M\"obius function.
\end{lemma}

Based on this generator count and the correlation function formula recalled in
\Cref{prop: prelimiting correlation function}, one obtains the following estimate for the limiting expected number of eigenvalues generating a fixed
finite extension. In the present paper, this result will be used to control the fixed-degree contribution from unramified extensions.

\begin{lemma}[{\cite[Theorem 1.10]{shen2026eigenvalues}}]\label{lem: arbitrary extension limit}
Let $K/\Q_p$ be any finite extension, and $p^f$ the order of its residue field. Then $\lim_{n \to \infty} \E[Z_{\mc{O}_K^{\new},n}]$ exists, is finite, and satisfies the estimate
     \begin{equation}
        \lim_{n \to \infty} \E[Z_{\mathcal{O}_K^{\new},n}] = ||\Disc_{K/\Q_p}|| \cdot \sum_{d|f} \mu\left(\frac{f}{d}\right)p^{d-f} + \mathcal{E}(K).
    \end{equation}
    Here $\mu$ is the M\"obius function, and the error term $\mathcal{E}(K)$ satisfies the bounds
    \begin{equation}
        -p^{-f}||\Disc_{K/\Q_p}|| < \mathcal{E}(K) < \frac{1+\tau(f)}{1-p^{-f}}p^{-f}||\Disc_{K/\Q_p}||
    \end{equation}
    where $\tau$ is the divisor function. 
\end{lemma}

\begin{lemma}\label{lem: unramified subextension}
Let $K/\Q_p$ be a finite extension with inertia degree $f$ and ramification index $e$. Let $\zeta_{p^f-1}$ be a primitive $(p^f-1)$-st root of unity, so that $\Q_p[\zeta_{p^f-1}]$ is the unique unramified extension over $\Q_p$ of degree $f$. Then, we have
$$\Q_p\subset\Q_p[\zeta_{p^f-1}]\subset K,$$
where the extension $K/\Q_p[\zeta_{p^f-1}]$ is totally ramified of degree $e$.
\end{lemma}

\begin{proof}
See Neukirch's textbook \cite[Proposition (7.5), Chapter II]{neukirch2013algebraic}.
\end{proof}

We will also use the following mass formula for finite extensions of
$\Q_p$, which allows us to sum the discriminant weights over extensions with fixed inertia degree and ramification index. This lemma is essentially due to Caruso \cite{caruso2022zeroes}; more precisely, it is a direct corollary of Caruso's formulation of Serre's mass formula.

\begin{lemma}\label{lem: serre mass}
Let $e,f\ge 1$ be positive integers. Let
$\zeta_{p^f-1}$ be a primitive $(p^f-1)$-st root of unity, so that
$\Q_p[\zeta_{p^f-1}]$ is the unique unramified extension of $\Q_p$ of degree $f$. Then
$$
\sum_K ||\Disc_{K/\Q_p}||=\frac{e}{p^{(e-1)f}},
$$
where the sum ranges over all totally ramified extensions $K/\Q_p[\zeta_{p^f-1}]$ of degree $e$, viewed as subfields of $\overline{\Q}_p$.
\end{lemma}

\begin{proof}
We apply \cite[Proposition 4.9]{caruso2022zeroes} with Caruso's base field
$F$ equal to our $\Q_p$. Then the residue-field cardinality in Caruso's notation is $q=p$. His proposition states that, for positive integers $r$ and $f$
with $f\mid r$,
$$
\sum_{K\in \mathrm{Ex}_{r,f}} ||\Disc_{K/\Q_p}||
=
\frac{r}{p^r}\cdot \frac{p^f}{f},
$$
where $\mathrm{Ex}_{r,f}$ denotes the set of embedded extensions
$K/\Q_p$ of degree $r$ and residual degree $f$.

In our situation, we take $r=ef$. An extension $K/\Q_p$ of degree $ef$ and residual degree $f$ has ramification index $e$. Equivalently, since $\Q_p[\zeta_{p^f-1}]$ is the unique unramified extension of
$\Q_p$ of degree $f$, and by \Cref{lem: unramified subextension} such an extension $K$ is precisely a totally ramified extension of $\Q_p[\zeta_{p^f-1}]$ of degree $e$. Thus the summation set in Caruso's proposition is exactly the
summation set in the statement of the lemma.

Substituting $r=ef$ into \cite[Proposition 4.9]{caruso2022zeroes} gives
$$
\sum_K ||\Disc_{K/\Q_p}||=\frac{ef}{p^{ef}}\cdot \frac{p^f}{f}=e p^{-(e-1)f}=
\frac{e}{p^{(e-1)f}}.
$$
This proves the lemma.
\end{proof}

\include*{Estimate_of_correlation_functions}

\section{The random matrix case}\label{sec: Proof}

In this section, we prove \Cref{thm: zeros degree under rn} and \Cref{thm: outside un} stated in the introduction.  We first
establish the large-extension estimates needed to control individual finite
extensions, then use them to prove the unramified result and the extension-degree
asymptotic.

\begin{prop}[Large extensions]\label{prop: large extension}
Let $K/\Q_p$ be a finite extension of degree $r=ef\le n$, where $e$ is the ramification index and $f$ is the inertia degree. Then, we have the upper bound
\begin{equation}\label{eq: large extension upper bound}
\E[Z_{\mc{O}_K^{\new},n}]\le\frac{||\Disc_{K/\Q_p}||}{1-p^{-f}}\prod_{i=0}^{r-1}(1-p^{-n+i}).
\end{equation}
Moreover, if $K/\Q_p$ is unramified, we have the lower bound
\begin{equation}\label{eq: large extension lower bound}
\E[Z_{\mc{O}_K^{\new},n}]\ge\frac{1-1000p^{-r/2}}{1-p^{-r}}\prod_{i=0}^{r-1}(1-p^{-n+i}).
\end{equation}
\end{prop}

The following lemma provides the determinant estimate needed for the lower
bound in \Cref{prop: large extension}.

\begin{lemma}\label{lem: determinant lower bound}
Suppose $Z\in\Z_p[x]$ is monic of degree $r$, and its residue $\bar Z\in\F_p[x]$ is irreducible. Let $A\in\Mat_n(\Z_p)$ be Haar distributed. Then, we have
$$\E||\det(Z(A))||\ge 1-\frac{1}{p^r-1}.$$
\end{lemma}

\begin{proof}
It is enough to prove a lower bound for the probability that
$\det(Z(A))$ is a $p$-adic unit. Indeed, since $Z(A)$ has entries in
$\Z_p$, we have
$$
||\det(Z(A))||=1
$$
whenever $\det(Z(A))\in \Z_p^\times$, and otherwise
$||\det(Z(A))||\ge 0$. Hence
$$
\E||\det(Z(A))||\ge\mathbf P\bigl(\det(Z(A))\in\Z_p^\times\bigr)=\mathbf P\bigl(\det(\bar Z(\bar A))\neq 0\bigr),
$$
where $\bar A$ is uniformly sampled from $\Mat_n(\F_p)$. We now bound the complementary probability. Suppose $\det(\bar Z(\bar A))=0$. Then there exists a nonzero vector $v\in \F_p^n$ such that
$$
\bar Z(\bar A)v=0.
$$
Since $\overline Z$ is irreducible and $\overline Z(\overline A)v=0$, the cyclic $\F_p[\overline A]$-submodule generated by $v$ is naturally a quotient of $\F_p[x]/(\overline Z)$. Since $v\ne 0$ and
$\F_p[x]/(\overline Z)$ is a field, this quotient is nonzero and hence is isomorphic to $\F_p[x]/(\overline Z)$. Therefore
$$
W:=\operatorname{span}_{\F_p}\{v,\overline A v,\ldots,\overline A^{r-1}v\}
$$
has dimension $r$, is $\overline A$-invariant, and the characteristic polynomial of $\overline A|_W$ is $\overline Z$. Therefore the number of matrices $\bar A\in\Mat_n(\F_p)$ such that $\det(\bar Z(\bar A))=0$ is at most the number of pairs $(W,\bar A)$ such that $W$ is an $r$-dimensional $\bar A$-invariant subspace, and the characteristic polynomial of $(\bar A|_W)$ is equal to $\bar Z$.

We count these pairs. The number of $r$-dimensional subspaces $W\subset \F_p^n$ is the Gaussian binomial coefficient
$$
\binom{n}{r}_p:=\frac{(p^n-1)\cdots(p^n-p^{r-1})}{(p^r-1)\cdots(p^r-p^{r-1})}.
$$
For a fixed $W$, the number of possible endomorphisms of $W$ with
characteristic polynomial $\bar Z$ is
$$
\frac{|\GL_r(\F_p)|}{p^r-1}.
$$
Indeed, since $\bar Z$ is irreducible of degree $r$, any such matrix is conjugate to the companion matrix of $\bar Z$, whose centralizer in $\GL_r(\F_p)$ is naturally
$$
(\F_p[x]/(\bar Z))^\times,
$$
of size $p^r-1$.

After choosing $\bar A|_W$, the remaining entries of $\bar A$ are arbitrary subject only to the condition that $W$ is invariant. With respect to a decomposition $\F_p^n=W\oplus W'$, such a matrix has block form
$$
\begin{pmatrix}
\bar A|_W & * \\
0 & *
\end{pmatrix}.
$$
Thus there are
$$
p^{r(n-r)}p^{(n-r)^2}=p^{n(n-r)}
$$
choices for the remaining blocks. Hence the number of bad matrices is at most
$$
\binom{n}{r}_p
\frac{|\GL_r(\F_p)|}{p^r-1}
p^{n(n-r)}.
$$
Dividing by the total number $p^{n^2}$ of matrices in $\Mat_n(\F_p)$, we obtain
$$
\mathbf P\bigl(\det(\bar Z(\bar A))=0\bigr)
\le\binom{n}{r}_p\frac{|\GL_r(\F_p)|}{p^r-1}p^{-nr}.
$$
Using
$$
\binom{n}{r}_p|\GL_r(\F_p)|=
(p^n-1)(p^n-p)\cdots(p^n-p^{r-1}),
$$
we get
$$
\mathbf P\bigl(\det(\bar Z(\bar A))=0\bigr)
\le
\frac{\prod_{i=0}^{r-1}(1-p^{i-n})}{p^r-1}
\le
\frac{1}{p^r-1}.
$$
Consequently,
$$
\E||\det(Z(A))||\ge\mathbf P\bigl(\det(\bar Z(\bar A)\ne 0)\ge
1-\frac{1}{p^r-1},
$$
which completes the proof.
\end{proof}

\begin{proof}[Proof of \Cref{prop: large extension}]
Apply \Cref{prop: prelimiting correlation function} with
$U=\mathcal O_K^{\new}$, we obtain
$$
\E[Z_{\mathcal O_K^{\new},n}]= \prod_{i=0}^{r-1}(1-p^{-n+i})\cdot
\frac{||\Disc_{K/\Q_p}||}{1-p^{-f}}
\int_{\mathcal O_K^{\new}}
\frac{\#(\Lambda\backslash \Mod_{\Z_p[x]})}
{\#(\mathcal O_K/\Z_p[x])^2}
\cdot
\E||\det(Z(A))||
\, dx,
$$
where $Z=Z_x$ is the minimal polynomial of $x$ over $\Q_p$, and
$A\in \Mat_{n-r}(\Z_p)$ is Haar distributed.

We first prove the upper bound.  For every $x\in \mathcal O_K^{\new}$, we have
$$
\frac{\#(\Lambda\backslash \Mod_{\Z_p[x]})}
{\#(\mathcal O_K/\Z_p[x])^2}
\le 1
$$
by \Cref{lem: estimate of orbital integral}. Moreover, $\E||\det(Z(A))||\le 1$, since $Z(A)$ has entries in $\Z_p$.
Also, $\mathcal O_K\setminus \mathcal O_K^{\new}$ is a finite union of
proper $\Q_p$-subspaces of $K$, and hence has Haar measure zero. Therefore
$$
\E[Z_{\mathcal O_K^{\new},n}]\le\frac{||\Disc_{K/\Q_p}||}{1-p^{-f}}\prod_{i=0}^{r-1}(1-p^{-n+i}).
$$

We next prove the lower bound when $K/\Q_p$ is unramified. In this case, we have $e=1,f=r$, and $||\Disc_{K/\Q_p}||=1$. Denote the generator set
$$
\mathcal{G}_K:=\{x\in \mathcal O_K:\mathcal O_K=\Z_p[x]\}.
$$
For $x\in \mathcal{G}_K$, we have $\#(\Lambda\backslash \Mod_{\Z_p[x]})=\#(\mathcal O_K/\Z_p[x])=1$, thus
$$
\frac{\#(\Lambda\backslash \Mod_{\Z_p[x]})}
{\#(\mathcal O_K/\Z_p[x])^2}=1.
$$
Moreover, the reduction of $Z=Z_x$ modulo $p$ is irreducible of degree $r$.
Hence, by \Cref{lem: determinant lower bound}, applied to the random Haar matrix
$A\in \Mat_{n-r}(\Z_p)$, we have
$$
\E||\det(Z(A))||\ge1-\frac{1}{p^r-1}.
$$
Consequently, \begin{equation}\label{eq: medium step lower bound large extension}
\E[Z_{\mathcal O_K^{\new},n}]\ge\prod_{i=0}^{r-1}(1-p^{-n+i})\cdot\frac{1}{1-p^{-r}}\cdot\mu_K(\mathcal{G}_K)
\left(1-\frac{1}{p^r-1}\right).
\end{equation}
It remains to estimate $\mu_K(\mathcal{G}_K)$, the proportion of the set $\mathcal{G}_K$ in $\mathcal{O}_K$ under the Haar probability measure. Since $K/\Q_p$ is unramified, an element $x\in \mathcal O_K$ satisfies $\mathcal O_K=\Z_p[x]$ if and only if its reduction $\bar x\in \F_{p^r}$ generates $\F_{p^r}$ over $\F_p$. Thus, by \Cref{lem: order of generator}, we have
$$
1-\mu_K(\mathcal{G}_K)=1-\frac{G_r}{p^r}=1-\sum_{d|r}\mu\left(\frac{r}{d}\right)p^d\le p^{-r}\sum_{\substack{\ell\mid r\\ \ell\ \mathrm{prime}}}
p^{r/\ell}.
$$
Using the elementary estimate 
$$
p^{-r}\sum_{\substack{\ell\mid r\\ \ell\ \mathrm{prime}}}p^{r/\ell}+\frac{1}{p^r-1}
\le1000p^{-r/2},
$$
we obtain from the generalized Bernoulli's inequality that
$$
\mu_K(\mathcal{G}_K)\left(1-\frac{1}{p^r-1}\right)\ge1-1000p^{-r/2}.
$$
Substituting this into the previous lower bound \eqref{eq: medium step lower bound large extension} gives
$$
\E[Z_{\mathcal O_K^{\new},n}]\ge\frac{1-1000p^{-r/2}}{1-p^{-r}}\prod_{i=0}^{r-1}(1-p^{-n+i}).
$$
Combining this with the upper bound proves the proposition.
\end{proof}

\begin{proof}[Proof of \Cref{thm: zeros degree under rn}]
Since every eigenvalue of an $n\times n$ matrix over $\Z_p$ has degree at most $n$
over $\Q_p$, and since the sets $\mathcal O_K^{\new}$ decompose the eigenvalues
according to the extension they generate, we have
$$
\sum_{[K:\Q_p]\le r_n}\E[Z_{\mathcal O_K^{\new},n}]+\sum_{r_n<[K:\Q_p]\le n}\E[Z_{\mathcal O_K^{\new},n}]=n.
$$
Therefore
$$
\sum_{[K:\Q_p]\le r_n}\E[Z_{\mathcal O_K^{\new},n}]-r_n=(n-r_n)-\sum_{r_n<[K:\Q_p]\le n}\E[Z_{\mathcal O_K^{\new},n}].
$$
It remains to analyze the contribution from extensions of degree larger than $r_n$.

We first show that the total contribution of non-unramified extensions of degree larger than $r_n$ tends to zero. Let $K/\Q_p$ have degree $r=ef$, where $e$ is the ramification index and $f$ is the inertia degree. If $K/\Q_p$ is not unramified, then $e\ge 2$. By
\Cref{prop: large extension},
$$
\E[Z_{\mathcal O_K^{\new},n}]\le\frac{||\Disc_{K/\Q_p}||}{1-p^{-f}}\prod_{i=0}^{r-1}(1-p^{-n+i})\le2||\Disc_{K/\Q_p}||.
$$
Here we used $\frac{1}{1-p^{-f}}\le 2$ and $\prod_{i=0}^{r-1}(1-p^{-n+i})\le 1$. By \Cref{lem: unramified subextension}, every extension $K/\Q_p$ with inertia degree $f$ and ramification index $e$ is a totally ramified extension of $\Q_p[\zeta_{p^f-1}]$ of degree $e$. Hence, by \Cref{lem: serre mass},
$$
\sum_{\substack{[K:\Q_p]=ef\\
\text{inertia degree } f\\
\text{ramification index } e}}||\Disc_{K/\Q_p}||=\frac{e}{p^{(e-1)f}}.
$$
Therefore the total contribution of non-unramified extensions of degree larger than $r_n$ is at most
$$
2\sum_{\substack{ef>r_n\\ e\ge 2}}\frac{e}{p^{(e-1)f}}\le 2\sum_{\substack{ef>r_n\\ e\ge 2}} e p^{-ef/2}\le2\sum_{m>r_n}m^2p^{-m/2}.
$$
Since $r_n\to\infty$, the right-hand side tends to zero. Hence
$$
\lim_{n\to\infty}\sum_{\substack{r_n<[K:\Q_p]\le n\\ K/\Q_p\text{ not unramified}}}
\E[Z_{\mathcal O_K^{\new},n}]=0.
$$

It remains to analyze the unramified extensions of degrees between $r_n+1$ and $n$. For each $1\le r\le n$, write
$$
K_r:=\Q_p[\zeta_{p^r-1}],
$$
the unique unramified extension of $\Q_p$ of degree $r$. From the preceding decomposition, we have
$$
\sum_{[K:\Q_p]\le r_n}\E[Z_{\mathcal O_K^{\new},n}]-r_n=\sum_{r=r_n+1}^{n}\left(1-\E[Z_{\mathcal O_{K_r}^{\new},n}]\right)
+o(1).
$$
Indeed, the term $1$ in each summand corresponds to the contribution of one possible degree, while the contribution from all non-unramified extensions in this range is $o(1)$.

We now evaluate the last sum. Fix $M\ge 1$. Since $n-r_n\to\infty$, for all sufficiently large $n$ we may split
$$
\sum_{r=r_n+1}^{n}
\left(1-\E[Z_{\mathcal O_{K_r}^{\new},n}]\right)=\sum_{r=r_n+1}^{n-M}\left(1-\E[Z_{\mathcal O_{K_r}^{\new},n}]\right)+\sum_{r=n-M+1}^{n}\left(1-\E[Z_{\mathcal O_{K_r}^{\new},n}]\right).
$$

We first bound the range $r_n<r\le n-M$. By \Cref{prop: large extension}, for the
unramified extension $K_r/\Q_p$ of degree $r$, we have
$$
\frac{1-1000p^{-r/2}}{1-p^{-r}}
\prod_{i=0}^{r-1}(1-p^{-n+i})\le\E[Z_{\mathcal O_{K_r}^{\new},n}]\le\frac{1}{1-p^{-r}}\prod_{i=0}^{r-1}(1-p^{-n+i}).
$$
We claim that these bounds imply
$$
\left|1-\E[Z_{\mathcal O_{K_r}^{\new},n}]
\right|\le C_p\left(p^{-r/2}+p^{-(n-r+1)}\right)
$$
for a constant $C_p$ depending only on $p$. Indeed, the upper bound gives
$$
\E[Z_{\mathcal O_{K_r}^{\new},n}]\le\frac{1}{1-p^{-r}},
$$
so the possible excess over $1$ is $O_p(p^{-r})$. For the lower bound, using
$$
1-\prod_{j=n-r+1}^{n}(1-p^{-j})\le\sum_{j=n-r+1}^{n}p^{-j}\le\frac{p^{-(n-r+1)}}{1-p^{-1}},
$$
we obtain
$$
\E[Z_{\mathcal O_{K_r}^{\new},n}]\ge 1-C_p p^{-r/2}-C_p p^{-(n-r+1)}.
$$
This proves the claimed estimate. Therefore
$$
\sum_{r=r_n+1}^{n-M}\left|1-\E[Z_{\mathcal O_{K_r}^{\new},n}]\right|\le C_p\sum_{r=r_n+1}^{\infty}p^{-r/2}+C_p\sum_{r=r_n+1}^{n-M}p^{-(n-r+1)}.
$$
The first sum is $O_p(p^{-r_n/2})$, and the second is $O_p(p^{-M})$. Hence
$$
\limsup_{n\to\infty}
\left|\sum_{r=r_n+1}^{n-M}\left(1-\E[Z_{\mathcal O_{K_r}^{\new},n}]\right)
\right|\le C_p' p^{-M}.
$$

It remains to analyze the last $M$ unramified extensions. Write $r=n-m+1$, where $1\le m\le M$. By \Cref{prop: large extension}, for each fixed $m$,
$$
\lim_{n\to\infty}
\E[Z_{\mathcal O_{K_{n-m+1}}^{\new},n}]=(p^{-m};p^{-1})_\infty.
$$
Indeed,
$$
\prod_{i=0}^{n-m}(1-p^{-n+i})=\prod_{j=m}^{n}(1-p^{-j})
\longrightarrow
\prod_{j=m}^{\infty}(1-p^{-j})=(p^{-m};p^{-1})_\infty,
$$
while the prefactor $\frac{1}{1-p^{-(n-m+1)}}$ and the lower-bound error term $1000p^{-(n-m+1)/2}$ tend to $1$ and $0$,
respectively. Consequently,
$$
\lim_{n\to\infty}
\sum_{r=n-M+1}^{n}
\left(1-\E[Z_{\mathcal O_{K_r}^{\new},n}]\right)=\sum_{m=1}^{M}
\left(1-(p^{-m};p^{-1})_\infty\right).
$$
Combining the two ranges, we obtain
$$
\limsup_{n\to\infty}\left|\sum_{r=r_n+1}^{n}\left(1-\E[Z_{\mathcal O_{K_r}^{\new},n}]\right)-\sum_{m=1}^{M}
\left(1-(p^{-m};p^{-1})_\infty\right)
\right|\le C_p' p^{-M}.
$$
Letting $M\to\infty$, and using the convergence of
$$
\sum_{m=1}^{\infty}\left(1-(p^{-m};p^{-1})_\infty\right),
$$
we conclude that
$$
\lim_{n\to\infty}\sum_{r=r_n+1}^{n}\left(1-\E[Z_{\mathcal O_{K_r}^{\new},n}]\right)=
\sum_{m=1}^{\infty}\left(1-(p^{-m};p^{-1})_\infty\right).
$$
Together with the fact that the non-unramified contribution is $o(1)$, this gives
$$
\lim_{n\to\infty}\left(\sum_{[K:\Q_p]\le r_n}\E[Z_{\mathcal O_K^{\new},n}]-r_n\right)=\sum_{m=1}^{\infty}\left(1-(p^{-m};p^{-1})_\infty\right),
$$
which proves the theorem.
\end{proof}

We now turn to the proof of \Cref{thm: outside un}.  We apply the same
two-sided estimates for unramified extensions, but now sum over all unramified
degrees $1\le r\le n$.  The following proposition gives a more precise
limiting formula for $n-\E[Z_{\Q_p^{\un},n}]$, and in particular proves the existence of the limiting expected number of eigenvalues outside
$\Q_p^{\un}$.

\begin{prop}\label{prop: unramified complement}
The sequence $n-\E[Z_{\Q_p^{\un},n}]$ converges to the limit
$$\lim_{n\rightarrow\infty}\left(n-\E[Z_{\Q_p^{\un},n}]\right)=\sum_{m=1}^\infty \left(1-\lim_{n\rightarrow\infty}\E[Z_{\mathcal O_{\Q_p[\zeta_{p^m-1}]}^{\new},n}]\right)+\sum_{m=1}^\infty\left(1-(p^{-m};p^{-1})_\infty\right).$$
\end{prop}

\begin{proof}
For each $1\le r\le n$, recall that
$\Q_p[\zeta_{p^r-1}]$ is the unique unramified extension of $\Q_p$ of degree $r$. Since an eigenvalue of an $n\times n$ matrix has degree at most $n$ over $\Q_p$,
we have
$$
\E[Z_{\Q_p^{\un},n}]=\sum_{r=1}^{n}
\E[Z_{\mathcal O_{\Q_p[\zeta_{p^r-1}]}^{\new},n}].
$$
Therefore
$$
n-\E[Z_{\Q_p^{\un},n}]=\sum_{r=1}^{n}\left(
1-\E[Z_{\mathcal O_{\Q_p[\zeta_{p^r-1}]}^{\new},n}]
\right).
$$

Fix $M\ge 1$. When $n>2M$, we can split this sum as
$$
\begin{aligned}
\sum_{r=1}^{n}\left(1-\E[Z_{\mathcal O_{\Q_p[\zeta_{p^r-1}]}^{\new},n}]
\right)&=\sum_{r=1}^{M}
\left(
1-\E[Z_{\mathcal O_{\Q_p[\zeta_{p^r-1}]}^{\new},n}]
\right) \\
&+\sum_{r=M+1}^{n-M}\left(
1-\E[Z_{\mathcal O_{\Q_p[\zeta_{p^r-1}]}^{\new},n}]\right) \\
&+\sum_{r=n-M+1}^{n}\left(1-\E[Z_{\mathcal O_{\Q_p[\zeta_{p^r-1}]}^{\new},n}]
\right),
\end{aligned}
$$
which is a combination of the fixed degree part, the middle range, and the large degree part.

We first consider the fixed degree part.  For each fixed $r$, the limit
$$
\lim_{n\to\infty}\E[Z_{\mathcal O_{\Q_p[\zeta_{p^r-1}]}^{\new},n}]
$$
exists by \Cref{lem: arbitrary extension limit}.  Hence
$$
\lim_{n\to\infty}\sum_{r=1}^{M}\left(
1-\E[Z_{\mathcal O_{\Q_p[\zeta_{p^r-1}]}^{\new},n}]\right)=\sum_{r=1}^{M}\left(
1-\lim_{n\to\infty}\E[Z_{\mathcal O_{\Q_p[\zeta_{p^r-1}]}^{\new},n}]
\right).
$$
We next show that the middle range gives no contribution after sending $M\to\infty$. By \Cref{prop: large extension}, for every $1\le r\le n$,
$$
\frac{1-1000p^{-r/2}}{1-p^{-r}}\prod_{i=0}^{r-1}(1-p^{-n+i})\le\E[Z_{\mathcal O_{\Q_p[\zeta_{p^r-1}]}^{\new},n}]\le\frac{1}{1-p^{-r}}\prod_{i=0}^{r-1}(1-p^{-n+i}).
$$
Moreover,
$$
\prod_{i=0}^{r-1}(1-p^{-n+i})=\prod_{j=n-r+1}^{n}(1-p^{-j}).
$$
If $M<r<n-M+1$, then $r\ge M+1$ and $n-r+1\ge M+1$.  Therefore
$$
1-\prod_{j=n-r+1}^{n}(1-p^{-j})\le\sum_{j=n-r+1}^{n}p^{-j}\le\frac{p^{-(n-r+1)}}{1-p^{-1}}.
$$
The upper and lower bounds from \Cref{prop: large extension} then imply
$$
\left|1-\E[Z_{\mathcal O_{\Q_p[\zeta_{p^r-1}]}^{\new},n}]\right|\le C_p\left(p^{-r/2}+p^{-(n-r+1)}\right)
$$
for a constant $C_p$ depending only on $p$. Hence
$$
\begin{aligned}
\sum_{r=M+1}^{n-M}\left|1-\E[Z_{\mathcal O_{\Q_p[\zeta_{p^r-1}]}^{\new},n}]\right|
&\le C_p\sum_{r=M+1}^{\infty}p^{-r/2}+C_p\sum_{r=M+1}^{n-M}p^{-(n-r+1)}  \\
&\le C_p'\,p^{-M/2}.
\end{aligned}
$$
Thus
$$
\lim_{M\to\infty}\limsup_{n\to\infty}\left|\sum_{r=M+1}^{n-M}\left(1-\E[Z_{\mathcal O_{\Q_p[\zeta_{p^r-1}]}^{\new},n}]\right)\right|=0.
$$

It remains to analyze the large degree part. Write $r=n-m+1$ with
$1\le m\le M$. Applying again the upper and lower bounds in \Cref{prop: large extension}, we obtain that for all fixed $1\le m\le M$,
$$
\lim_{n\to\infty}\E[Z_{\mathcal O_{\Q_p[\zeta_{p^{n-m+1}-1}]}^{\new},n}]=(p^{-m};p^{-1})_\infty.
$$
Therefore
$$
\lim_{n\to\infty}\sum_{r=n-M+1}^{n}\left(
1-\E[Z_{\mathcal O_{\Q_p[\zeta_{p^r-1}]}^{\new},n}]\right)=\sum_{m=1}^{M}\left(
1-(p^{-m};p^{-1})_\infty\right).
$$

We now combine the three ranges more carefully. For fixed $M$, the preceding
estimates show that
\begin{multline*}
\limsup_{n\to\infty}\Bigg|\left(n-\E[Z_{\Q_p^{\un},n}]\right)-\sum_{r=1}^{M}\left(1-\lim_{n\to\infty}
\E[Z_{\mathcal O_{\Q_p[\zeta_{p^r-1}]}^{\new},n}]
\right)-\sum_{m=1}^{M}\left(1-(p^{-m};p^{-1})_\infty\right)\Bigg|\\
\le C_p' p^{-M/2}.
\end{multline*}
Indeed, the first sum comes from the fixed degree range $1\le r\le M$, the second sum comes from the large degree range $n-M+1\le r\le n$, and the only remaining contribution is the middle range, which is bounded by $C_p'p^{-M/2}$. Sending $M\to\infty$ proves that
$$
\lim_{n\to\infty}\left(n-\E[Z_{\Q_p^{\un},n}]\right)=\sum_{m=1}^{\infty}\left(1-\lim_{n\to\infty}
\E[Z_{\mathcal O_{\Q_p[\zeta_{p^m-1}]}^{\new},n}]\right)+\sum_{m=1}^{\infty}
\left(1-(p^{-m};p^{-1})_\infty\right),
$$
as claimed.
\end{proof}

The expression obtained in \Cref{prop: unramified complement} immediately gives
the following explicit estimate.

\begin{corollary}\label{cor: explicit estimate unramified complement}
We have
$$
\lim_{n\rightarrow\infty}\E[Z_{\bar\Q_p\backslash\Q_p^{\un},n}]=\lim_{n\to\infty}\left(n-\E[Z_{\Q_p^{\un},n}]\right)\le\frac{p(4p-3)}{(p-1)^3}.
$$
\end{corollary}

\begin{proof}
By \Cref{prop: unramified complement},
$$\lim_{n\rightarrow\infty}\left(n-\E[Z_{\Q_p^{\un},n}]\right)=\sum_{m=1}^\infty \left(1-\lim_{n\rightarrow\infty}\E[Z_{\mathcal O_{\Q_p[\zeta_{p^m-1}]}^{\new},n}]\right)+\sum_{m=1}^\infty\left(1-(p^{-m};p^{-1})_\infty\right).$$
We now bound the two series on the right-hand side.  First, since the extension $\Q_p[\zeta_{p^m-1}]/\Q_p$ is unramified of degree $m$, we have $||\Disc_{\Q_p[\zeta_{p^m-1}]/\Q_p}||=1$. Therefore, by \Cref{lem: arbitrary extension limit},
$$
\lim_{n\to\infty}
\E[Z_{\mathcal O_{\Q_p[\zeta_{p^m-1}]}^{\new},n}]=\sum_{d\mid m}\mu\left(\frac{m}{d}\right)p^{d-m}+\mathcal{E}(\Q_p[\zeta_{p^m-1}]),
$$
where
$$
-p^{-m}<\mathcal{E}(\Q_p[\zeta_{p^m-1}])<\frac{1+\tau(m)}{1-p^{-m}}p^{-m}.
$$
Since the term $d=m$ in the main sum is equal to $1$, we have
$$
\left|1-\lim_{n\to\infty}
\E[Z_{\mathcal O_{\Q_p[\zeta_{p^m-1}]}^{\new},n}]\right|\le\sum_{\substack{d\mid m\\ d<m}}p^{d-m}+\frac{1+\tau(m)}{1-p^{-m}}p^{-m}.
$$
Therefore
$$
\lim_{n\to\infty}\left(n-\E[Z_{\Q_p^{\un},n}]\right)\le\sum_{m=1}^{\infty}\left(1-(p^{-m};p^{-1})_\infty\right)+\sum_{m=1}^{\infty}\left(\sum_{\substack{d\mid m\\ d<m}}p^{d-m}+\frac{1+\tau(m)}{1-p^{-m}}p^{-m}
\right).
$$
This proves the more precise displayed estimate.

It remains to give the simpler closed-form bound. We use three elementary
estimates. First, following the generalized Bernoulli's inequality,
$$
1-(p^{-m};p^{-1})_\infty=1-\prod_{j=m}^{\infty}(1-p^{-j})\le\sum_{j=m}^{\infty}p^{-j}.
$$
Hence, we have
\begin{equation}\label{eq: first bound}
\sum_{m=1}^{\infty}\left(1-(p^{-m};p^{-1})_\infty\right)\le\sum_{m=1}^{\infty}\sum_{j=m}^{\infty}p^{-j}=\sum_{j=1}^{\infty}j p^{-j}=\frac{p}{(p-1)^2}.
\end{equation}

Second,
$$
\sum_{m=1}^{\infty}\sum_{\substack{d\mid m\\ d<m}}p^{d-m}\le\sum_{d=1}^{\infty}\sum_{k=2}^{\infty}p^{d-kd}=\sum_{d=1}^{\infty}\frac{p^{-d}}{1-p^{-d}}.
$$
Since
$$
\sum_{d=1}^{\infty}\frac{p^{-d}}{1-p^{-d}}=
\sum_{d=1}^{\infty}\frac{1}{p^d-1}\le\frac{p}{(p-1)^2},
$$
we get
\begin{equation}\label{eq: second bound}
\sum_{m=1}^{\infty}\sum_{\substack{d\mid m\\ d<m}}p^{d-m}\le\frac{p}{(p-1)^2}.
\end{equation}

Third, since $\frac{1}{1-p^{-m}}\le \frac{p}{p-1}$, we have
$$
\sum_{m=1}^{\infty}\frac{1+\tau(m)}{1-p^{-m}}p^{-m}\le\frac{p}{p-1}\left(
\sum_{m=1}^{\infty}p^{-m}+\sum_{m=1}^{\infty}\tau(m)p^{-m}\right).
$$
Now $\sum_{m=1}^{\infty}p^{-m}=\frac{1}{p-1}$, and
$$
\sum_{m=1}^{\infty}\tau(m)p^{-m}=\sum_{a=1}^{\infty}\sum_{b=1}^{\infty}p^{-ab}=\sum_{a=1}^{\infty}\frac{p^{-a}}{1-p^{-a}}\le\frac{p}{(p-1)^2}.
$$
Therefore
\begin{equation}\label{eq: third bound}
\sum_{m=1}^{\infty}\frac{1+\tau(m)}{1-p^{-m}}p^{-m}\le\frac{p}{p-1}\left(\frac{1}{p-1}+\frac{p}{(p-1)^2}\right)=\frac{p(2p-1)}{(p-1)^3}.
\end{equation}

Combining the three estimates \eqref{eq: first bound}, \eqref{eq: second bound}, and \eqref{eq: third bound} gives
$$
\lim_{n\to\infty}\left(n-\E[Z_{\Q_p^{\un},n}]\right)\le\frac{p}{(p-1)^2}+\frac{p}{(p-1)^2}+\frac{p(2p-1)}{(p-1)^3}=\frac{p(4p-3)}{(p-1)^3}.
$$
This proves the corollary.
\end{proof}

\begin{proof}[Proof of \Cref{thm: outside un}]
By \Cref{prop: unramified complement}, the sequence
$n-\E[Z_{\Q_p^{\un},n}]$ converges as $n\to\infty$, and its limit is identified by decomposing the unramified contribution into fixed degree, middle-degree, and large degree ranges.  Since
\begin{equation*}
Z_{\overline{\Q}_p\setminus\Q_p^{\un},n}=
n-Z_{\Q_p^{\un},n},
\end{equation*}
this proves the existence of the limit in \Cref{thm: outside un}. The explicit upper bound $\frac{p(4p-3)}{(p-1)^3}$ is exactly given in
\Cref{cor: explicit estimate unramified complement}.

It remains to show that the limit is positive. Let $K/\Q_p$ be any ramified
quadratic extension. Then $K$ is not contained in $\Q_p^{\un}$, and every element of $\mathcal O_K^{\new}$ generates the ramified extension $K$ over $\Q_p$. Hence
\begin{equation*}
\mathcal O_K^{\new}\subset\overline{\Q}_p\setminus \Q_p^{\un}.
\end{equation*}
Therefore, for every $n$,
\begin{equation*}
\E[Z_{\overline{\Q}_p\setminus\Q_p^{\un},n}]\ge\E[Z_{\mathcal O_K^{\new},n}].
\end{equation*}
By \cite[Theorem 1.10]{shen2026eigenvalues}, the limit $
\lim_{n\to\infty}\E[Z_{\mathcal O_K^{\new},n}]$ exists and is positive. Consequently,
\begin{equation*}
\lim_{n\to\infty}
\E[Z_{\overline{\Q}_p\setminus\Q_p^{\un},n}]\ge\lim_{n\to\infty}\E[Z_{\mathcal O_K^{\new},n}]>0.
\end{equation*}
This proves the positivity of the limit and completes the proof.
\end{proof}

\section{The random polynomial case}\label{sec: The random polynomial case}

In this section, we prove the polynomial analogues of the main results from the random matrix case. We first recall the correlation function formula from Caruso \cite{caruso2022zeroes} for random Haar polynomials, then use it to prove uniform estimates for finite extensions and finally establish \Cref{thm: poly zeros degree under rn} and \Cref{thm: poly outside un}.

\begin{prop}\label{prop: prelimiting correlation function_poly}
Let $K/\Q_p$ be a finite extension of degree $r\le n$. Let $U \subset K^{\new}$ be any measurable set. Then, we have
$$\E[Z_{U,n}^{\poly}]=\int_U\rho^{(n),\poly}_{K}(x)dx,$$
where
\begin{equation}
\rho_{K}^{(n),\poly}(x)=
\begin{cases}
\frac{||\Disc_{K/\Q_p}||}{\#(\mc{O}_{K}/\Z_p[x])} \cdot \int_{\Omega_{n-r}(\Z_p)}||P(x)||^rdP & x\in\mathcal{O}_K^{\new}\\ 
||x||^{-2r}\cdot\rho_{K}^{(n),\poly}(x^{-1}) & x\in K^{\new}\backslash\mc{O}_K^{\new}
\end{cases}.
\end{equation}
Here, the integration (and definition of measurability of $U$) is with respect to the Haar probability measure on $K$, normalized so that $\mc{O}_K$ has measure $1$. 
\end{prop}

\begin{proof}
We apply Caruso's correlation function formalism \cite[Definition 2.3 and
Theorem 2.6]{caruso2022zeroes} with his base field $F$ equal to our $\Q_p$.
Let us first translate the notation.  In Caruso's notation, $q$ denotes the
cardinality of the residue field of $F$, so in our setting $q=p$.  His ring of
integers $\mathcal O_F$ is our $\Z_p$, and his space $\Omega_n$ of polynomials
of degree at most $n$ with coefficients in $\mathcal O_F$ is exactly our
$\Omega_n(\Z_p)$.  The probability measure $\mu_n$ on $\Omega_n$ in Caruso's
paper is the product Haar measure on the coefficients, which is precisely the
probability measure used here.  Finally, for a finite extension $K/\Q_p$,
Caruso's Haar measure $\lambda_K$ on $K$ is normalized by
$\lambda_K(\mathcal O_K)=1$, which is the normalization of $dx$ used in the
statement.

Caruso defines, for every finite extension $K/F$ and every $n$, a density
function $\rho_{K,n}:K\to\mathbb R_{\ge 0}$.  In our notation this is the
function denoted by $\rho_K^{(n),\poly}$.  We now recall its explicit form in
the special case $F=\Q_p$.  Let $r=[K:\Q_p]$.  If $x\in \mathcal O_K^{\new}$,
then \cite[Definition 2.3]{caruso2022zeroes} gives
\begin{equation}
\rho_K^{(n),\poly}(x)=\frac{||\Disc_{K/\Q_p}||}{\#(\mathcal O_K/\Z_p[x])}\int_{\Omega_{n-r}(\Z_p)} ||P(x)||^r\,dP.
\end{equation}
Here Caruso's discriminant factor $D_K$ is the $p$-adic norm of the discriminant of $K/\Q_p$, namely $||\Disc_{K/\Q_p}||$, and
$\mathcal O_F[x]$ in his notation is $\Z_p[x]$ in ours.

It remains to explain the formula outside the unit ball. If $x\in K^{\new}\setminus \mathcal O_K^{\new}$, then $x^{-1}\in
\mathcal O_K^{\new}$. \cite[Definition 2.3]{caruso2022zeroes} defines the density at $x$ by the homographic transformation $x\mapsto x^{-1}$, giving
\begin{equation}
\rho_K^{(n),\poly}(x)=||x||^{-2r}\rho_K^{(n),\poly}(x^{-1}).
\end{equation}
This is exactly the formula given in the statement of the proposition.

We now turn to the integral formula.  In Caruso's notation, an element $x\in K$ is new in $K$ if $F[x]=K$.  Since $F=\Q_p$ here, this is exactly the condition $x\in K^{\new}$ in our notation. For an open subset $U\subset K$, Caruso defines $Z_{U,n}^{\new}$ to be the number of roots of a random polynomial in $\Omega_n$ which lie in $U$ and are new in $K$.  Therefore, when $U\subset K^{\new}$, this random variable is exactly our $Z_{U,n}^{\poly}$. \cite[Theorem 2.6]{caruso2022zeroes} then gives
\begin{equation}
\E[Z_{U,n}^{\poly}]=\int_U \rho_K^{(n),\poly}(x)\,dx
\end{equation}
for open subsets $U\subset K^{\new}$.

Finally, the same identity holds for every measurable subset $U\subset K^{\new}$ considered in this paper. Indeed, both sides define measures on $K^{\new}$: the left-hand side is the expected counting measure of new roots, and the right-hand side is the measure with density $\rho_K^{(n),\poly}$ with respect to Haar measure. Since the two measures agree on open subsets by Caruso's theorem, and Haar measure on the local field $K$ is regular, they agree on all measurable subsets. This proves the proposition.
\end{proof}

\begin{prop}\label{prop: bound for not unramified_poly}
Let $K/\Q_p$ be a finite extension of degree $r=ef\le n$, where $e$ is the ramification index and $f$ is the inertia degree. Then, we have the upper bound
$$\E[Z_{K^{\new},n}^{\poly}]\le(1+4p^{-f})||\Disc_{K/\Q_p}||.$$
\end{prop}

\begin{proof}
We use the estimate from Caruso \cite{caruso2022zeroes} for the correlation function under the polynomial setting. We specialize $F=\Q_p$, so his $q$ is our $p$. Moreover, his $||D_K||$ is $||\Disc_{K/\Q_p}||$ under our notation,  his $\rho_{K,n}$ is our $\rho_{K}^{(n),\poly}$, and his $\rho_n(K)$ is our $\E[Z_{K^{\new},n}^{\poly}]$.

Now, under the notation of the present paper, \cite[Corollary 4.6]{caruso2022zeroes} gives, for $n\ge 2r-1$,
$$
\frac{\E[Z_{K^{\new},n}^{\poly}]}{||\Disc_{K/\Q_p}||}-\frac{p^f}{p^f+1}\cdot \frac{G_f}{p^f}\le 4p^{-f}.
$$
Here, the definition of $G_f$ follows from \eqref{eq: G_r}. It is clear that $G_f\le p^f$, hence
$$
\E[Z_{K^{\new},n}^{\poly}]\le(1+4p^{-f})||\Disc_{K/\Q_p}||
$$
when $n\ge 2r-1$.

It remains to treat the range $r\le n<2r-1$. By the monotonicity of Caruso's correlation functions \cite[Theorem B.5]{caruso2022zeroes}, for every $x\in K^{\new}$ we have
$$
\rho^{(n),\poly}_K(x)\le\rho^{(2r-1),\poly}_K(x).
$$
Integrating over $K$ gives
$$
\int_K \rho^{(n),\poly}_K(x)\,dx\le\int_K \rho^{(2r-1),\poly}_K(x)\,dx.
$$
Applying the estimate already proved with $n=2r-1$, we get
$$
\E[Z^{\poly}_{K^{\new},n}]=\int_K \rho^{(n),\poly}_K(x)\,dx\le(1+4p^{-f})||\Disc_{K/\Q_p}||.
$$
This completes the proof.
\end{proof}

\begin{prop}\label{prop: bounds for unramified extension_poly}
When $K=\Q_p[\zeta_{p^r-1}]$ and $r\le n\le 2r-1$, we have
$$\frac{p^{n+1}-p^r}{p^{n+1}-1}\cdot\sum_{d|r}\mu\left(\frac{r}{d}\right)p^{d-r}-p^{-r}\le\E[Z_{\Q_p[\zeta_{p^r-1}]^{\new},n}^{\poly}]\le\frac{p^{n+1}-p^r}{p^{n+1}-1}\cdot\sum_{d|r} \mu\left(\frac{r}{d}\right)p^{d-r}+4p^{-r}.$$
Here $\mu$ is the M\"obius function. Moreover, when $n\ge 2r-1$, we have
$$\E[Z_{\Q_p[\zeta_{p^r-1}]^{\new},n}^{\poly}]=\E[Z_{\Q_p[\zeta_{p^r-1}]^{\new},2r-1}^{\poly}].$$
\end{prop}

\begin{proof}
The notation transfer is the same as in the proof of \Cref{prop: bound for not unramified_poly}. We first assume that $r\le n\le 2r-1$. Since $\Q_p[\zeta_{p^r-1}]/\Q_p$ is unramified, we have $||\Disc_{\Q_p[\zeta_{p^r-1}]/\Q_p}||=1$. By \cite[Theorem 4.7]{caruso2022zeroes}, specialized to the present notation, we have
$$
-p^{-r}\le\E[Z^{\poly}_{\Q_p[\zeta_{p^r-1}]^{\new},n}]-\frac{p^{n+1}-p^r}{p^{n+1}-1}\cdot \frac{G_r}{p^r}\le4p^{-r},
$$
Based on the expression of $G_r$ in \Cref{lem: order of generator}, this becomes
$$
\frac{p^{n+1}-p^r}{p^{n+1}-1}\cdot\sum_{d|r}\mu\left(\frac{r}{d}\right)p^{d-r}-p^{-r}\le
\E[Z^{\poly}_{\Q_p[\zeta_{p^r-1}]^{\new},n}]\le
\frac{p^{n+1}-p^r}{p^{n+1}-1}\cdot\sum_{d|r}\mu\left(\frac{r}{d}\right)p^{d-r}+4p^{-r}.
$$
This proves the desired two-sided estimate in the range $r\le n\le 2r-1$.

It remains to prove the stabilization statement for $n\ge 2r-1$. By Caruso's monotonicity theorem, the polynomial correlation function is independent of $n$ once $n\ge 2r-1$ \cite[Theorem B]{caruso2022zeroes}. Equivalently, for every $x\in \Q_p[\zeta_{p^r-1}]^{\new}$ and every $n\ge 2r-1$, we have
$$
\rho^{(n),\poly}_{\Q_p[\zeta_{p^r-1}]}(x)=\rho^{(2r-1),\poly}_{\Q_p[\zeta_{p^r-1}]}(x).
$$
Integrating over $\Q_p[\zeta_{p^r-1}]$ and using \Cref{prop: prelimiting correlation function_poly}, we obtain
$$
\E[Z^{\poly}_{\Q_p[\zeta_{p^r-1}]^{\new},n}]=\int_{\Q_p[\zeta_{p^r-1}]} \rho^{(n),\poly}_{\Q_p[\zeta_{p^r-1}]}(x)\,dx=\int_{\Q_p[\zeta_{p^r-1}]} \rho^{(2r-1),\poly}_{\Q_p[\zeta_{p^r-1}]}(x)\,dx=\E[Z^{\poly}_{\Q_p[\zeta_{p^r-1}]^{\new},2r-1}].
$$
This completes the proof.
\end{proof}

\begin{proof}[Proof of \Cref{thm: poly zeros degree under rn}]
Since a Haar-random polynomial in $\Omega_n(\Z_p)$ has degree $n$ almost surely, it has exactly $n$ roots in $\overline{\Q}_p$, counted with multiplicity. Moreover, the sets $K^{\new}$ decompose the roots according to the finite extension of $\Q_p$ which they generate. Hence
$$
\sum_{[K:\Q_p]\le r_n}\E[Z_{K^{\new},n}^{\poly}]+\sum_{r_n<[K:\Q_p]\le n}\E[Z_{K^{\new},n}^{\poly}]=n.
$$
Therefore
$$
\sum_{[K:\Q_p]\le r_n}\E[Z_{K^{\new},n}^{\poly}]-r_n=(n-r_n)-
\sum_{r_n<[K:\Q_p]\le n}\E[Z_{K^{\new},n}^{\poly}].
$$
It remains to analyze the contribution from extensions of degree larger than $r_n$.

We first show that the total contribution of non-unramified extensions of degree larger than $r_n$ tends to zero. Let $K/\Q_p$ have degree $r=ef$, where $e$ is the ramification index and $f$ is the inertia degree. If $K/\Q_p$ is not unramified, then $e\ge 2$. By \Cref{prop: bound for not unramified_poly}, we have
$$
\E[Z_{K^{\new},n}^{\poly}]\le(1+4p^{-f})||\Disc_{K/\Q_p}||\le5||\Disc_{K/\Q_p}||.
$$
By \Cref{lem: unramified subextension}, every extension $K/\Q_p$ with inertia degree $f$ and ramification index $e$ is a totally ramified extension of $\Q_p[\zeta_{p^f-1}]$ of degree $e$. Hence, by \Cref{lem: serre mass},
$$
\sum_{\substack{[K:\Q_p]=ef\\
\text{inertia degree } f\\
\text{ramification index } e}}
||\Disc_{K/\Q_p}||=\frac{e}{p^{(e-1)f}}.
$$
Therefore the total contribution of non-unramified extensions of degree larger than $r_n$ is at most
$$
5\sum_{\substack{ef>r_n\\ e\ge 2}}\frac{e}{p^{(e-1)f}}\le 5\sum_{\substack{ef>r_n\\ e\ge 2}} e p^{-ef/2}\le 5\sum_{m>r_n}m^2p^{-m/2}.
$$
Since $r_n\to\infty$, the right-hand side tends to zero. Hence
$$
\lim_{n\to\infty}\sum_{\substack{r_n<[K:\Q_p]\le n\\ K/\Q_p\text{ not unramified}}}\E[Z_{K^{\new},n}^{\poly}]=0.
$$

It remains to analyze the unramified extensions of degrees between $r_n+1$ and $n$. For each $1\le r\le n$, write
$$
K_r:=\Q_p[\zeta_{p^r-1}],
$$
the unique unramified extension of $\Q_p$ of degree $r$. From the preceding decomposition, we have
$$
\sum_{[K:\Q_p]\le r_n}\E[Z_{K^{\new},n}^{\poly}]-r_n=\sum_{r=r_n+1}^{n}\left(1-\E[Z_{K_r^{\new},n}^{\poly}]\right)+o(1).
$$
We now evaluate the last sum. Fix $M\ge 1$. Since $n-r_n\to\infty$, for all sufficiently large $n$ we can split
$$
\sum_{r=r_n+1}^{n}\left(1-\E[Z_{K_r^{\new},n}^{\poly}]\right)=\sum_{r=r_n+1}^{n-M}\left(1-\E[Z_{K_r^{\new},n}^{\poly}]\right)+\sum_{r=n-M+1}^{n}\left(1-\E[Z_{K_r^{\new},n}^{\poly}]\right).
$$

We first bound the range $r_n<r\le n-M$. We claim that for all $1\le r\le n$,
$$
\left|1-\E[Z_{K_r^{\new},n}^{\poly}]\right|
\le C_p\left(p^{-r/2}+p^{-(n-r+1)}\right),
$$
where $C_p$ is a constant depending only on $p$. Indeed, first suppose $r\le n\le 2r-1$. By \Cref{prop: bounds for unramified extension_poly},
$$
-p^{-r}\le\E[Z_{K_r^{\new},n}^{\poly}]-\frac{p^{n+1}-p^r}{p^{n+1}-1}\cdot \frac{G_r}{p^r}\le4p^{-r}.
$$
Moreover, we have
$$
1-\frac{G_r}{p^r}=1-\sum_{d|r}\mu\left(\frac{r}{d}\right)p^d\le p^{-r}\sum_{\substack{\ell\mid r\\ \ell\ \mathrm{prime}}}p^{r/\ell}\le C_p p^{-r/2}.
$$
Also,
$$
1-\frac{p^{n+1}-p^r}{p^{n+1}-1}=\frac{p^r-1}{p^{n+1}-1}\le C_p p^{-(n-r+1)}.
$$
Combining these estimates gives
$$
\left|1-\E[Z_{K_r^{\new},n}^{\poly}]\right|\le
C_p\left(p^{-r/2}+p^{-(n-r+1)}\right)
$$
in the range $r\le n\le 2r-1$.

If $n\ge 2r-1$, then \Cref{prop: bounds for unramified extension_poly} gives the
stabilization
$$
\E[Z_{K_r^{\new},n}^{\poly}]=\E[Z_{K_r^{\new},2r-1}^{\poly}].
$$
Applying the preceding estimate with $n=2r-1$, we get
$$
\left|1-\E[Z_{K_r^{\new},n}^{\poly}]
\right|\le C_p p^{-r/2}.
$$
This is also bounded by
$$
C_p\left(p^{-r/2}+p^{-(n-r+1)}\right).
$$
This proves the claimed uniform estimate. Therefore
$$
\sum_{r=r_n+1}^{n-M}\left|1-\E[Z_{K_r^{\new},n}^{\poly}]\right|\le
C_p\sum_{r=r_n+1}^{\infty}p^{-r/2}+C_p\sum_{r=r_n+1}^{n-M}p^{-(n-r+1)}.
$$
The first sum is $O_p(p^{-r_n/2})$, and the second is $O_p(p^{-M})$. Hence
$$
\limsup_{n\to\infty}\left|\sum_{r=r_n+1}^{n-M}\left(1-\E[Z_{K_r^{\new},n}^{\poly}]\right)\right|\le C_p' p^{-M}.
$$
It remains to analyze the last $M$ unramified extensions. Write $r=n-m+1$, where $1\le m\le M$. Then, for fixed $m$ and all sufficiently large $n$, we have $r\le n\le 2r-1$. By \Cref{prop: bounds for unramified extension_poly},
$$
-p^{-r}\le\E[Z_{K_r^{\new},n}^{\poly}]-\frac{p^{n+1}-p^r}{p^{n+1}-1}\cdot\frac{G_r}{p^r}\le 4p^{-r}.
$$
As $n\to\infty$, we have $r=n-m+1\to\infty$, and hence $\frac{G_r}{p^r}\to 1$. Moreover,
$$
\frac{p^{n+1}-p^r}{p^{n+1}-1}=
\frac{p^{n+1}-p^{n-m+1}}{p^{n+1}-1}\to1-p^{-m}.
$$
Therefore, for every fixed $1\le m\le M$,
$$
\lim_{n\to\infty}\E[Z_{K_{n-m+1}^{\new},n}^{\poly}]=1-p^{-m}.
$$
Consequently,
$$
\lim_{n\to\infty}\sum_{r=n-M+1}^{n}\left(1-\E[Z_{K_r^{\new},n}^{\poly}]\right)=\sum_{m=1}^{M}p^{-m}.
$$

Combining the two ranges, we obtain
$$
\limsup_{n\to\infty}\left|\sum_{r=r_n+1}^{n}\left(1-\E[Z_{K_r^{\new},n}^{\poly}]\right)-\sum_{m=1}^{M}p^{-m}\right|\le
C_p' p^{-M}.
$$
Letting $M\to\infty$, and using $\sum_{m=1}^{\infty}p^{-m}=\frac{1}{p-1}$,
we conclude that
$$
\lim_{n\to\infty}\sum_{r=r_n+1}^{n}\left(1-\E[Z_{K_r^{\new},n}^{\poly}]\right)=\frac{1}{p-1}.
$$
Together with the fact that the non-unramified contribution is $o(1)$, this gives
$$
\lim_{n\to\infty}\left(\sum_{[K:\Q_p]\le r_n}\E[Z_{K^{\new},n}^{\poly}]-r_n\right)=\frac{1}{p-1}.
$$
This proves the theorem.
\end{proof}

\begin{prop}\label{prop: unramified complement_poly}
The sequence $n-\E[Z_{\Q_p^{\un},n}^{\poly}]$ converges to the limit
$$\lim_{n\rightarrow\infty}\left(n-\E[Z_{\Q_p^{\un},n}^{\poly}]\right)=\sum_{m=1}^\infty \left(1-\E[Z_{\Q_p[\zeta_{p^m-1}]^{\new},2m-1}^{\poly}]\right)+\frac{1}{p-1}.$$
\end{prop}

\begin{proof}
For each $1\le r\le n$, recall that $\Q_p[\zeta_{p^r-1}]$ is the unique unramified extension of $\Q_p$ of degree $r$. Since a polynomial of degree $n$ has exactly $n$ roots over $\overline{\Q}_p$, counted with multiplicity, we have
$$
\E[Z^{\poly}_{\Q_p^{\un},n}]=\sum_{r=1}^n
\E[Z^{\poly}_{\Q_p[\zeta_{p^r-1}]^{\new},n}].
$$
Therefore
$$
n-\E[Z^{\poly}_{\Q_p^{\un},n}]=\sum_{r=1}^n
\left(1-\E[Z^{\poly}_{\Q_p[\zeta_{p^r-1}]^{\new},n}]\right).
$$

Fix $M\ge 1$. When $n>2M$, we split the sum into three ranges:
\begin{align}
\begin{split}
\sum_{r=1}^n\left(1-\E[Z^{\poly}_{\Q_p[\zeta_{p^r-1}]^{\new},n}]\right)&=\sum_{r=1}^M\left(
1-\E[Z^{\poly}_{\Q_p[\zeta_{p^r-1}]^{\new},n}]\right)\\
&+\sum_{r=M+1}^{n-M}
\left(1-\E[Z^{\poly}_{\Q_p[\zeta_{p^r-1}]^{\new},n}]\right)\\
&+\sum_{r=n-M+1}^{n}\left(
1-\E[Z^{\poly}_{\Q_p[\zeta_{p^r-1}]^{\new},n}]\right).
\end{split}
\end{align}
We call these the fixed degree part, the middle range, and the large degree part.

We first consider the fixed degree part. For every fixed $1\le r\le M$, \Cref{prop: bounds for unramified extension_poly}
implies that the polynomial correlation function, and hence also the integral over
$\Q_p[\zeta_{p^r-1}]^{\new}$, is independent of $n$ once $n\ge 2r-1$. Thus, for all sufficiently
large $n$,
$$
\E[Z^{\poly}_{\Q_p[\zeta_{p^r-1}]^{\new},n}]=\E[Z^{\poly}_{\Q_p[\zeta_{p^r-1}]^{\new},2r-1}].
$$
Therefore
$$
\lim_{n\to\infty}\sum_{r=1}^M\left(1-\E[Z^{\poly}_{\Q_p[\zeta_{p^r-1}]^{\new},n}]\right)=\sum_{r=1}^M\left(
1-\E[Z^{\poly}_{\Q_p[\zeta_{p^r-1}]^{\new},2r-1}]\right).
$$

We next show that the middle range gives no contribution after sending $M\to\infty$.
We claim that, for all $1\le r\le n$,
$$
\left|1-\E[Z^{\poly}_{\Q_p[\zeta_{p^r-1}]^{\new},n}]\right|\le C_p\left(p^{-r/2}+p^{-(n-r+1)}\right),
$$
where $C_p$ is a constant depending only on $p$.

Indeed, first suppose $r\le n\le 2r-1$. By \Cref{prop: bounds for unramified extension_poly},
$$
-p^{-r}\le\E[Z^{\poly}_{\Q_p[\zeta_{p^r-1}]^{\new},n}]-\frac{p^{n+1}-p^r}{p^{n+1}-1}\cdot \frac{G_r}{p^r}\le4p^{-r}.
$$
Following \Cref{lem: order of generator}, we have
$$
1-\frac{G_r}{p^r}=1-\sum_{d|r}\mu\left(\frac{r}{d}\right)p^d\le p^{-r}\sum_{\substack{\ell\mid r\\ \ell\ \mathrm{prime}}}p^{r/\ell}\le C_p p^{-r/2}.
$$
Also,
$$
1-\frac{p^{n+1}-p^r}{p^{n+1}-1}=\frac{p^r-1}{p^{n+1}-1}\le C_p p^{-(n-r+1)}.
$$
Combining these estimates gives
$$
\left|1-\E[Z^{\poly}_{\Q_p[\zeta_{p^r-1}]^{\new},n}]\right|\le C_p\left(p^{-r/2}+p^{-(n-r+1)}\right)
$$
in the range $r\le n\le 2r-1$.

If $n\ge 2r-1$, then \Cref{prop: bounds for unramified extension_poly} gives
$$
\E[Z^{\poly}_{\Q_p[\zeta_{p^r-1}]^{\new},n}]=\E[Z^{\poly}_{\Q_p[\zeta_{p^r-1}]^{\new},2r-1}].
$$
Applying the preceding estimate with $n=2r-1$ gives
$$
\left|1-\E[Z^{\poly}_{\Q_p[\zeta_{p^r-1}]^{\new},n}]\right|\le C_p p^{-r/2}.
$$
This is also bounded by
$$
C_p\left(p^{-r/2}+p^{-(n-r+1)}\right).
$$
This proves the claimed uniform estimate. Therefore
$$
\sum_{r=M+1}^{n-M}\left|1-\E[Z^{\poly}_{\Q_p[\zeta_{p^r-1}]^{\new},n}]\right|\le C_p\sum_{r=M+1}^{\infty}p^{-r/2}+C_p\sum_{r=M+1}^{n-M}p^{-(n-r+1)}.
$$
The right-hand side is bounded by $C_p' p^{-M/2}$. Hence
$$
\lim_{M\to\infty}\limsup_{n\to\infty}
\left|\sum_{r=M+1}^{n-M}\left(
1-\E[Z^{\poly}_{\Q_p[\zeta_{p^r-1}]^{\new},n}]\right)\right|=0.
$$

It remains to analyze the large degree part. Write $r=n-m+1$, where $1\le m\le M$. Then, for fixed $m$ and all sufficiently large $n$, we have $r\le n\le 2r-1$. By \Cref{prop: bounds for unramified extension_poly},
$$
-p^{-r}\le\E[Z^{\poly}_{\Q_p[\zeta_{p^r-1}]^{\new},n}]-\frac{p^{n+1}-p^r}{p^{n+1}-1}\cdot \frac{G_r}{p^r}\le 4p^{-r}.
$$
As $n\to\infty$, we have $r=n-m+1\to\infty$, and hence $
\frac{G_r}{p^r}\to 1$. Moreover,
$$
\frac{p^{n+1}-p^r}{p^{n+1}-1}=\frac{p^{n+1}-p^{n-m+1}}{p^{n+1}-1}\to 1-p^{-m}.
$$
Therefore, for every fixed $1\le m\le M$,
$$
\lim_{n\to\infty}
\E[Z^{\poly}_{\Q_p[\zeta_{p^{n-m+1}-1}]^{\new},n}]=1-p^{-m}.
$$
Consequently,
$$
\lim_{n\to\infty}\sum_{r=n-M+1}^{n}\left(
1-\E[Z^{\poly}_{\Q_p[\zeta_{p^r-1}]^{\new},n}]\right)=\sum_{m=1}^M p^{-m}.
$$

Combining the three ranges, we obtain that for each fixed $M$,
$$
\limsup_{n\to\infty}\left|\left(n-\E[Z^{\poly}_{\Q_p^{\un},n}]\right)-\sum_{r=1}^M\left(1-\E[Z^{\poly}_{\Q_p[\zeta_{p^r-1}]^{\new},2r-1}]\right)-\sum_{m=1}^M p^{-m}\right|\le C_p' p^{-M/2}.
$$
By the same uniform estimate with $n=2r-1$, we have
$$
\left|1-\E[Z^{\poly}_{\Q_p[\zeta_{p^r-1}]^{\new},2r-1}]\right|\le C_p p^{-r/2}.
$$
Hence the series
$$
\sum_{r=1}^\infty\left(1-\E[Z^{\poly}_{\Q_p[\zeta_{p^r-1}]^{\new},2r-1}]\right)
$$
converges absolutely. Since $\sum_{m=1}^{\infty}p^{-m}=\frac{1}{p-1}$, sending $M\to\infty$ gives
$$
\lim_{n\to\infty}\left(n-\E[Z^{\poly}_{\Q_p^{\un},n}]\right)=\sum_{m=1}^\infty\left(1-\E[Z^{\poly}_{\Q_p[\zeta_{p^m-1}]^{\new},2m-1}]\right)+\frac{1}{p-1}.
$$
This proves the proposition.
\end{proof}

\begin{proof}[Proof of \Cref{thm: poly outside un}]
Since a Haar-random polynomial in $\Omega_n(\Z_p)$ has degree $n$ almost surely, it has exactly $n$ roots in $\overline{\Q}_p$, counted with multiplicity. Therefore
$$
Z_{\overline{\Q}_p\setminus \Q_p^{\un},n}^{\poly}=n-Z_{\Q_p^{\un},n}^{\poly}
$$
almost surely. Hence \Cref{prop: unramified complement_poly} gives the existence of the limit and the identity
$$
\lim_{n\to\infty}\E[Z_{\overline{\Q}_p\setminus \Q_p^{\un},n}^{\poly}]= \sum_{m=1}^\infty\left(1-\E[Z_{\Q_p[\zeta_{p^m-1}]^{\new},2m-1}^{\poly}]\right)+\frac{1}{p-1}.
$$
We first prove the upper bound for this limit.

For each $m\ge 1$, apply \Cref{prop: bounds for unramified extension_poly} with $r=m$ and $n=2m-1$. Since $\frac{p^{2m}-p^m}{p^{2m}-1}=\frac{p^m}{p^m+1}$, \Cref{prop: bounds for unramified extension_poly} gives
$$
\E[Z_{\Q_p[\zeta_{p^m-1}]^{\new},2m-1}^{\poly}]\ge\frac{p^m}{p^m+1}\sum_{d|m}\mu\left(\frac{m}{d}\right)p^{d-m}-p^{-m}.
$$
Therefore
$$
1-\E[Z_{\Q_p[\zeta_{p^m-1}]^{\new},2m-1}^{\poly}]\le 1-\frac{p^m}{p^m+1}\sum_{d|m}\mu\left(\frac{m}{d}\right)p^{d-m}+p^{-m}.
$$
We decompose the right-hand side as
$$
1-\frac{p^m}{p^m+1}\sum_{d|m}\mu\left(\frac{m}{d}\right)p^{d-m}+p^{-m}=\left(1-\frac{p^m}{p^m+1}\right)+\frac{p^m}{p^m+1}\left(1-\sum_{d|m}\mu\left(\frac{m}{d}\right)p^{d-m}\right)+p^{-m}.
$$
Since $1-\frac{p^m}{p^m+1}=\frac{1}{p^m+1}\le p^{-m}$ and
$$
1-\sum_{d|m}\mu\left(\frac{m}{d}\right)p^{d-m}=-\sum_{\substack{d|m\\ d<m}}\mu\left(\frac{m}{d}\right)p^{d-m}
\le\sum_{\substack{d|m\\ d<m}}p^{d-m},
$$
we get
$$
1-\E[Z_{\Q_p[\zeta_{p^m-1}]^{\new},2m-1}^{\poly}]\le2p^{-m}+\sum_{\substack{d|m\\ d<m}}p^{d-m}.
$$
Summing over $m\ge 1$, we obtain
$$
\sum_{m=1}^\infty\left(1-\E[Z_{\Q_p[\zeta_{p^m-1}]^{\new},2m-1}^{\poly}]\right)\le2\sum_{m=1}^\infty p^{-m}+\sum_{m=1}^\infty\sum_{\substack{d|m\\ d<m}}p^{d-m}.
$$
The first sum is $2\sum_{m=1}^\infty p^{-m}=\frac{2}{p-1}$. For the second sum, write $m=kd$ with $k\ge 2$. Then
$$
\sum_{m=1}^\infty\sum_{\substack{d|m\\ d<m}}p^{d-m}\le\sum_{d=1}^\infty\sum_{k=2}^\infty p^{d-kd}=\sum_{d=1}^\infty \frac{p^{-d}}{1-p^{-d}}=\sum_{d=1}^\infty \frac{1}{p^d-1}.
$$
Using $\frac{1}{p^d-1}\le \frac{p^{1-d}}{p-1}$, we get
$$
\sum_{d=1}^\infty \frac{1}{p^d-1}\le\frac{1}{p-1}\sum_{d=1}^\infty p^{1-d}
=\frac{p}{(p-1)^2}.
$$
Consequently,
$$
\sum_{m=1}^\infty\left(1-\E[Z_{\Q_p[\zeta_{p^m-1}]^{\new},2m-1}^{\poly}]\right)\le\frac{2}{p-1}+\frac{p}{(p-1)^2}.
$$
Substituting this into the expression from \Cref{prop: unramified complement_poly} gives
$$
\lim_{n\to\infty}\E[Z_{\overline{\Q}_p\setminus \Q_p^{\un},n}^{\poly}]\le\frac{1}{p-1}
+\frac{2}{p-1}+\frac{p}{(p-1)^2}=\frac{4p-3}{(p-1)^2}.
$$

It remains to show that the limit is positive. Let $K/\Q_p$ be any ramified quadratic extension. Then $K$ is not contained in $\Q_p^{\un}$, and every element of $K^{\new}$ generates the ramified extension $K$ over $\Q_p$.  Therefore, for every $n$,
$$
\E[Z_{\overline{\Q}_p\setminus\Q_p^{\un},n}^{\poly}]\ge\E[Z_{K^{\new},n}^{\poly}].
$$
By the totally ramified quadratic case of \cite[Theorem C]{caruso2022zeroes}, for $n\ge 3$ we have
$$
\E[Z_{K^{\new},n}^{\poly}]=||\Disc_{K/\Q_p}||\cdot\frac{p^2(p^2+1)}{p^4+p^3+p^2+p+1}.
$$
This quantity is strictly positive. Consequently,
$$
\lim_{n\to\infty}\E[Z_{\overline{\Q}_p\setminus \Q_p^{\un},n}^{\poly}]\ge||\Disc_{K/\Q_p}||
\cdot\frac{p^2(p^2+1)}{p^4+p^3+p^2+p+1}>0.
$$
This proves the positivity of the limit and completes the proof.
\end{proof}

\bibliographystyle{plain}
\bibliography{references.bib}

\end{document}